\documentclass{article}
\usepackage{amssymb}
\usepackage{amsfonts}
\usepackage{amsmath}

\newtheorem{theorem}{Theorem}

\newtheorem{condition}[theorem]{Condition}

\newtheorem{corollary}[theorem]{Corollary}

\newtheorem{definition}[theorem]{Definition}

\newtheorem{lemma}[theorem]{Lemma}

\newtheorem{proposition}[theorem]{Proposition}
\newtheorem{remark}[theorem]{Remark}

\newenvironment{proof}[1][Proof]{\noindent\textbf{#1.} }{\ \rule{0.5em}{0.5em}}
\input{tcilatex}

\begin{document}

\title{Isoperimetry and Rough Path Regularity}
\author{Peter Friz\thanks{%
Corresponding author. Department of Pure Mathematics and Mathematical
Statistics, University of Cambridge. Email: P.K.Friz@statslab.cam.ac.uk. }\ 
\thanks{%
Leverhulme Fellow.} \and Harald Oberhauser\thanks{%
Department of Pure Mathematics and Mathematical Statistics, University of
Cambridge. }}
\maketitle

\begin{abstract}
Optimal sample path properties of stochastic processes often involve
generalized H\"{o}lder- or variation norms. Following a classical result of
Taylor, the exact variation of Brownian motion is measured in terms of $\psi
\left( x\right) \equiv $ $x^{2}/\log \log \left( 1/x\right) $ near $0+$.
Such $\psi $-variation results extend to classes of processes with values in
abstract metric spaces. (No Gaussian or Markovian properties are assumed.)
To establish integrability properties of the $\psi $-variation we turn to a
large class of Gaussian rough paths (e.g.\ Brownian motion and L\'{e}vy's
area viewed as a process in a Lie group) and prove Gaussian integrability
properties using Borell's inequality on abstract Wiener spaces. The interest
in such results is that they are compatible with rough path theory and yield
certain sharp regularity and integrability properties (for iterated
Stratonovich integrals, for example) which would be difficult to obtain
otherwise. At last, $\psi $-variation is identified as robust regularity
property of solutions to (random) rough differential equations beyond
semimartingales.
\end{abstract}

\section{Introduction}

Optimal sample path properties of stochastic processes often involve
generalized H\"{o}lder- or variation norms. For Brownian motion these
results are classical and known as \textit{L\'{e}vy's modulus} and \textit{%
Taylor's variation} regularity respectively. Given an arbitrary stochastic
process $X\left( \omega \right) :\left[ 0,1\right] \rightarrow \left(
E,d\right) $, we give a criterion (condition \ref{ConditionExp} below) which
implies (via Garsia-Rodemich-Rumsey) Gauss tails for a L\'{e}vy-type $%
\varphi $-modulus "norm\footnote{%
There is no linear space here but the analogy to well-known (semi)norms is
strong enough to convince us to use the word "norm". }". In the very same
setting, we show that $X$ has a.s.\ finite $\psi $-variation "norm" of
Taylor-type. In general, $X$ need not be of a $\psi ^{-1}$-modulus and one
needs careful probabilistic arguments, adapted from Taylor \cite{Taylor72}
to our setting in theorem \ref{TaylorVar}; these arguments are ill-suited to
extract any integrability of the $\psi $-variation norms.

If $X$ is a real-valued Gaussian process and the function $\psi :\left[
0,\infty \right) \rightarrow \left[ 0,\infty \right) $ is reasonably behaved
(in particular, convex) one deals with a genuine (semi)norm and a classical
result of Fernique implies (Gaussian) integrability. A more interesting
example - and the motivation of the present work - are $\mathbb{R}^{d}$%
-valued centered \textit{Gaussian rough paths} i.e.\ processes enhanced with
a stochastic area process (such as \textit{L\'{e}vy's area} in the case of
Brownian motion). We view processes enhanced with their stochastic area as
processes with values in $G^{2}\left( \mathbb{R}^{d}\right) \cong \mathbb{R}%
^{d}\oplus so\left( d\right) $, the step-$2$ nilpotent group equipped with
Carnot-Caratheodory (CC) metric. The resulting H\"{o}lder and variation
spaces (which play a fundamental r\^{o}le in \textit{rough path theory})
have norms involving stochastic areas and hence do \textit{not} allow us to
use Fernique's result. Integrability properties of Wiener-It\^{o} chaos
(e.g.\ \cite[Thm 4.1]{ledoux-1996}) allow to go a bit further but ultimately
fail to deal with the non-linear structure of the H\"{o}lder and variation
"norms" of rough paths. To wit, the fine regularity properties required in
rough path theory rely crucially on the cancellations on the right-hand-side
of (\ref{AreaNotLinear}) below\footnote{%
With focus on H\"{o}lder regularity, a "linear" brute-force approach using
integrability of Banach-space valued Wiener-It\^{o} chaos leads to a Gauss
tail of 
\begin{equation*}
\sup_{0\leq s<t\leq 1}\frac{\left\vert A_{s,t}\right\vert ^{1/2}}{\left\vert
t-s\right\vert ^{\alpha }}
\end{equation*}%
for $\alpha <1/4$, compare with the discussion preceeding (\ref%
{AreaNotLinear}). In order to apply rough path theory, however, it is
crucial to take $\alpha >1/3$ so that $\left[ p\right] =2$ with $p=1/\alpha $
and $\alpha \in \left( 1/3,1/2\right) $.}. We overcome this difficulties
with Borell's isoperimetric inequality which leads us to a \textit{%
generalized Fernique estimate} (Theorem \ref{ThGeneralizedFernique}). As it
may well be useful in other situations, we state and prove it in its natural
setting of abstract Wiener spaces.\bigskip

Keeping the recalls on rough path theory to a minimum, we remind the reader
that rough paths take values in nilpotent groups with path regularity tied
to the degree of nilpotency. For instance, Brownian motion and L\'{e}vy's
area have (w.r.t.\ CC\ metric) finite $p=\left( 2+\varepsilon \right) $%
-variation and so one has to work in the group of step $\left[ p\right] =2$
nilpotency. The main result in rough path theory, due to T. Lyons, is that
higher iterated integrals, stochastic integrals of $1$-forms and the It\^{o}
map (i.e.\ the solution map to stochastic differential equations) all become
continuous and deterministic functions of \textit{Brownian motion and L\'{e}%
vy's area.} The study of $\psi $-variation is then natural from several
points of view;

\begin{itemize}
\item[(i)] it allows us to establish the definite variational regularity of
Gaussian rough paths;

\item[(ii)] $\psi $-variation regularity is the key to optimal regularity
results for the coefficients of differential equations driven by rough paths
(forthcoming work by A.M.Davie, not discussed here);

\item[(iii)] we shall see that rough path estimates, usually stated in $p$%
-variation, are valid in suitable $\psi $-variation.
\end{itemize}

It may be helpful to state some of the implications of this work without
using too much language of rough path theory and with focus on the simplest
possible Gaussian rough path: Brownian motion and L\'{e}vy area. For
instance, we have novel regularity and integrability properties of L\'{e}vy
area increments defined as 
\begin{equation*}
A_{s,t}=\frac{1}{2}\left( \int_{s}^{t}\left( B_{u}-B_{s}\right) \mathrm{d}%
\tilde{B}_{u}-\int_{s}^{t}\left( \tilde{B}_{u}-\tilde{B}_{s}\right) \mathrm{d%
}B_{u}\right)
\end{equation*}%
where $\left( B,\tilde{B}\right) $ is a $2$-dimensional standard Brownian
motion. The regularity result in the following theorem \ref{ThmLevyArea}
below must not be confused with the essentially trivial statement that $%
t\mapsto A_{0,t}$ has a.s.\ finite $\psi $-variation\footnote{$\left(
A_{0,t}:t\geq 0\right) $ is a continuous martingale, hence a time-change of
Brownian motion. Conclude with Taylor's variation regularity.}. The
situation is analogue to the subtle $\left\vert A_{s,t}\right\vert
^{1/2}\sim \left\vert t-s\right\vert ^{1/2-\varepsilon }$ versus the simple\ 
$\left\vert A_{t}-A_{s}\right\vert \sim \left\vert t-s\right\vert
^{1/2-\varepsilon }$, based on the cancellation taking place in the right
hand side of%
\begin{equation}
A_{s,t}=A_{t}-A_{s}-\frac{1}{2}\left( B_{s}\left( \tilde{B}_{t}-\tilde{B}%
_{s}\right) -\tilde{B}_{s}\left( B_{t}-B_{s}\right) \right) .
\label{AreaNotLinear}
\end{equation}

\begin{theorem}[Optimal regularity and integrability of L\'{e}vy's Area]
\label{ThmLevyArea}Set%
\begin{eqnarray*}
V_{\psi \circ \sqrt{\cdot }\text{-var;}\left[ 0,1\right] }\left( A\right)
&\equiv &\sup_{\left( t_{i}\right) \subset \left[ 0,1\right] }\sum_{i}\psi
\left( \left\vert A_{t_{i},t_{i+1}}\right\vert ^{1/2}\right) , \\
\left\vert A\right\vert _{\psi \circ \sqrt{\cdot }\text{-var;}\left[ 0,1%
\right] } &\equiv &\inf \left\{ \varepsilon >0:V_{\psi \circ \sqrt{\cdot }%
\text{-var;}\left[ 0,1\right] }\left( A/\varepsilon ^{2}\right) \leq
1\right\} .
\end{eqnarray*}%
Let $\psi :\left[ 0,\infty \right) \rightarrow \left[ 0,\infty \right) ,$ $%
\psi \left( 0\right) =0$ be continuous, strictly increasing, onto and such
that $\psi \left( x\right) =x^{2}/\log \log \left( 1/x\right) $ near $0+$.
Then $\left\vert A\right\vert _{\psi \circ \sqrt{\cdot }\text{-var;}\left[
0,1\right] }<\infty $ a.s.\ and has a Gauss tail. Moreover, this $\psi $%
-variation is optimal in the sense that for any $\tilde{\psi}:\left[
0,\infty \right) \rightarrow \left[ 0,\infty \right) $ with $%
\lim_{x\rightarrow 0}\tilde{\psi}\left( x\right) /\psi \left( x\right)
=\infty $ we have almost surely%
\begin{equation*}
V_{\tilde{\psi}\circ \sqrt{\cdot }\text{-var;}\left[ 0,1\right] }\left(
A\right) =\left\vert A\right\vert _{\tilde{\psi}\circ \sqrt{\cdot }\text{%
-var;}\left[ 0,1\right] }=+\infty .
\end{equation*}
\end{theorem}

\begin{remark}
\textbf{(Warning)} We are not saying that L\'{e}vy's area, say $\left\vert
A_{0,1}\right\vert ,$ has a Gauss tail. What we are saying is that
quantities of type $\left\vert A_{0,1}\right\vert ^{1/2}$ have a Gauss tail.
The square-root arises naturally if one seeks homogenous path space norm
which deal simultanously with paths and area and are crucial for our
application in rough path theory.
\end{remark}

For example, theorem \ref{ThmLevyArea} sharpens the well-known statement 
\cite{lyons-qian-02} that for $p>2$ (which corresponds to $\psi \left(
x\right) =x^{p}$ above) 
\begin{equation*}
\left\vert A\right\vert _{p\circ \sqrt{\cdot }\text{-var;}\left[ 0,1\right]
}=\left( \sup_{D\subset \left[ 0,1\right] }\sum_{i:t_{i}\in D}\left\vert
A_{t_{i},t_{i+1}}\right\vert ^{p/2}\right) ^{1/p}<\infty \text{ a.s.}
\end{equation*}%
(This variational regularity of L\'{e}vy area increments is precisely what
allows to use rough path analysis in conjunction with Brownian motion.)

As further application (and with regard to item (iii)\ above) rough path
estimates are compatible with $\psi $-variation. For instance, Brownian
motion and L\'{e}vy area can be enhanced with any number of higher iterated
Stratonovich integrals and the resulting process has finite $\psi $%
-variation with Gaussian integrability of the associated homogenous norms,
discussed in section \ref{SubSectionroughpathinpsivariation}. Similar
arguments apply to sample path regularity of solution to stochastic
differential equations in the rough path sense. When these are
semimartingales (as is usually the case in Stratonovich theory) a.s.\ finite 
$\psi $-variation is seen just as for $t\mapsto A_{0,t}$ above. On the other
hand, our results imply that such regularity results are "robust" beyond
semimartingales and apply to large classes of differential equations driven
by Gaussian signals (including but far from restricted to fractional
Brownian motion).

\section{\protect\bigskip A Generalized Fernique Theorem}

Let $\left( \mathbb{B},\left\vert \cdot \right\vert \right) $ be a real,
separable Banach space equipped with its Borel $\sigma $-algebra $\mathfrak{B%
}$ and a centered Gaussian measure $\mu $. A famous result by X. Fernique
states that $\left\vert \cdot \right\vert _{\ast }\mu $ has a Gauss tail;
more precisely,%
\begin{equation*}
\int \exp \left( \eta \left\vert x\right\vert ^{2}\right) \mathrm{d}\mu
\left( x\right) <\infty \text{ if }\eta <\frac{1}{2\sigma ^{2}},
\end{equation*}%
where 
\begin{equation}
\sigma :=\sup_{\xi \in B^{\ast },\left\vert \xi \right\vert _{\mathcal{B}%
^{\ast }}=1}\left( \,\int \left\langle \xi ,x\right\rangle ^{2}\mathrm{d}\mu
\left( x\right) \right) ^{1/2}<\infty ,  \label{DefSigmaAbstractWienerSpace}
\end{equation}%
and this condition on $\eta $ is sharp. See \cite[Thm 4.1]{ledoux-1996} for
instance. We recall the notion of a \textit{reproducing kernel Hilbert space}
$\mathcal{H}$, continuously embedded in $\mathbb{B}$,%
\begin{equation*}
\left\vert h\right\vert \leq \sigma \left\vert h\right\vert _{\mathcal{H}%
}\forall h\in \mathcal{H}\text{,}
\end{equation*}%
so that $\left( \mathbb{B},\mathcal{H},\mu \right) $ is an abstract Wiener
space in the sense of L. Gross. (The standard example to have in mind is the
Wiener-space $C_{0}\left( \left[ 0,1\right] ,\mathbb{R}\right) $ equipped
with Wiener measure; then $\mathcal{H}$ is the space of all absolutely
continuous paths with $h\left( 0\right) =0$ and $\dot{h}\in L^{2}\left( %
\left[ 0,1\right] \right) $.) We can then cite Borell's inequality, e.g.\ 
\cite[Theorem 4.3]{ledoux-1996}.

\begin{theorem}
\label{ThBorellTIS}Let $\left( \mathbb{B},\mathcal{H},\mu \right) $ be an
abstract Wiener space and $A\subset E$ a measurable Borel set with $\mu
\left( A\right) >0$. Take $a\in (-\infty ,\infty ]$ such that%
\begin{equation*}
\mu \left( A\right) =\int_{-\infty }^{a}\frac{1}{\sqrt{2\pi }}e^{-x^{2}/2}%
\mathrm{d}x=:\Phi \left( a\right) .
\end{equation*}%
Then, if $\mathcal{K}$ denotes the unit ball in $\mathcal{H}$ and $\mu
_{\ast }$ stands for the inner measure\footnote{%
Measurability of the so-called Minkowski sum $A+r\mathcal{K}$ is a delicate
topic. Use of the inner measure bypasses this issue and is not restrictive
in applications.} associated to $\mu $,%
\begin{equation}
\mu _{\ast }\left( A+r\mathcal{K}\right) =\mu _{\ast }\left\{ x+rh:x\in
A,\,h\in \mathcal{K}\right\} \geq \Phi \left( a+r\right) .
\label{EqBorellInequality}
\end{equation}
\end{theorem}

The reader should observe that the following theorem reduces to the usual
Fernique result when applied to the Banach norm on $\mathbb{B}$.

\begin{theorem}[Generalized Fernique Estimate]
\label{ThGeneralizedFernique}Let $\left( \mathbb{B},\mathcal{H},\mu \right) $
be an abstract Wiener space. Assume $f:\mathbb{B}\rightarrow \mathbb{R\cup }%
\left\{ -\infty ,\infty \right\} $ is a measurable map and $N\subset \mathbb{%
B}$ a null-set such that for all $b\notin N$ 
\begin{equation}
\left\vert f\left( b\right) \right\vert <\infty  \label{EqfFinite}
\end{equation}%
and for some positive constant $c$, 
\begin{equation}
\forall h\in \mathcal{H}\text{: }\left\vert f\left( b\right) \right\vert
\leq c\left\{ \left\vert \left( f\left( b-h\right) \right) \right\vert
+\sigma \left\vert h\right\vert _{\mathcal{H}}\right\} .
\label{Control_H_translate}
\end{equation}%
Then, with the definition of $\sigma $ given in (\ref%
{DefSigmaAbstractWienerSpace}),%
\begin{equation*}
\int \exp \left( \eta \left\vert f\left( b\right) \right\vert ^{2}\right) 
\mathrm{d}\mu \left( b\right) <\infty \text{ \ if }\eta <\frac{1}{%
2c^{2}\sigma ^{2}}.
\end{equation*}
\end{theorem}

\begin{proof}
We have for all $b\notin N$\ and all $h\in r\mathcal{K}$, where $\mathcal{K}$
denotes the unit ball of $\mathcal{H}$ and $r>0,$%
\begin{eqnarray*}
\left\{ b:\left\vert f\left( b\right) \right\vert \leq M\right\} &\supset
&\left\{ b:c\left( \left\vert f\left( b-h\right) \right\vert +\sigma
\left\vert h\right\vert _{\mathcal{H}}\right) \leq M\right\} \\
&\supset &\left\{ b:c\left( \left\vert f\left( b-h\right) \right\vert
+\sigma r\right) \leq M\right\} \\
&=&\left\{ b+h:\left\vert f\left( b\right) \right\vert \leq M/c-\sigma
r\right\} \text{.}
\end{eqnarray*}%
Since $h\in r\mathcal{K}$ was arbitrary,%
\begin{eqnarray*}
\left\{ b:\left\vert f\left( b\right) \right\vert \leq M\right\} &\supset
&\cup _{h\in r\mathcal{K}}\left\{ b+h:\left\vert f\left( b\right)
\right\vert \leq M/c-\sigma r\right\} \\
&=&\left\{ b:\left\vert f\left( b\right) \right\vert \leq M/c-\sigma
r\right\} +r\mathcal{K}
\end{eqnarray*}%
and we see that%
\begin{eqnarray*}
\mu \left[ \left\vert f\left( b\right) \right\vert \leq M\right] &=&\mu
_{\ast }\left[ \left\vert f\left( b\right) \right\vert \leq M\right] \\
&\geq &\mu _{\ast }\left( \left\{ b:\left\vert f\left( b\right) \right\vert
\leq M/c-\sigma r\right\} +r\mathcal{K}\right)
\end{eqnarray*}%
We can take $M=\left( 1+\varepsilon \right) c\sigma r$ and obtain%
\begin{equation*}
\mu \left[ \left\vert f\left( b\right) \right\vert \leq \left( 1+\varepsilon
\right) c\sigma r\right] \geq \mu _{\ast }\left( \left\{ b:\left\vert
f\left( b\right) \right\vert \leq \varepsilon \sigma r\right\} +r\mathcal{K}%
\right)
\end{equation*}%
Keeping $\varepsilon $ fixed, take $r\geq r_{0}$ where $r_{0}$ is chosen
large enough such that 
\begin{equation*}
\mathbb{\mu }\left[ \left\{ b:\left\vert f\left( b\right) \right\vert \leq
\varepsilon \sigma r_{0}\right\} \right] >0.
\end{equation*}%
Letting $\Phi $ denote the distribution function of a standard Gaussian, it
follows from Borell's inequality\ that%
\begin{equation*}
\mu \left[ \left\vert f\left( b\right) \right\vert \leq \left( 1+\varepsilon
\right) c\sigma r\right] \geq \Phi \left( a+r\right)
\end{equation*}%
for some $a>-\infty $. Equivalently, 
\begin{equation*}
\mu \left[ \left\vert f\left( b\right) \right\vert \geq x\right] \leq \bar{%
\Phi}\left( a+\frac{x}{\left( 1+\varepsilon \right) c\sigma }\right)
\end{equation*}%
with $\bar{\Phi}\equiv 1-\Phi $ and using $\bar{\Phi}\left( z\right)
\lesssim \exp \left( -z^{2}/2\right) $ this we see that this implies 
\begin{equation*}
\int \exp \left( \eta \left\vert f\left( b\right) \right\vert ^{2}\right) 
\mathrm{d}\mu \left( b\right) <\infty
\end{equation*}%
provided%
\begin{equation*}
\eta <\frac{1}{2}\left( \frac{1}{\left( 1+\varepsilon \right) c\sigma }%
\right) ^{2}.
\end{equation*}%
Sending $\varepsilon \rightarrow 0$ finishes the proof.
\end{proof}

\section{Regularity of stochastic processes\label{Section Regularity}}

Sharp sample path properties for stochastic processes often require
generalized H\"{o}lder- or variation norms. Using the following definition, L%
\'{e}vy's modulus for Brownian motion is captured by $\varphi _{2,1}$-H\"{o}%
lder regularity, Taylor's variation regularity corresponds to generalized $%
\psi _{2,2}$-variation.\ (Granted continuity and strictly monotonicity of $%
\varphi $ and $\psi $, only the behaviour near zero matters.)

\begin{definition}
\label{DefPsiAndPhi}Given $x>0$ we define\footnote{%
All $\psi ^{\prime }$s and $\varphi ^{\prime }$s below extend continuously
to $0$ with $\psi \left( 0\right) =0,\,\varphi \left( 0\right) =0$.} 
\begin{eqnarray*}
\varphi _{p,1}\left( x\right) &=&x^{1/p}\sqrt{\log _{1}x}\text{ and }\psi
_{p,1}\left( x\right) =\left\vert \frac{x}{\sqrt{\log _{1}x}}\right\vert ^{p}%
\text{,} \\
\text{where }\log _{1}\left( x\right) &=&\left\{ 
\begin{array}{cc}
\log \frac{1}{x} & \text{, for }x\leq e^{-1} \\ 
1 & \text{, otherwise}%
\end{array}%
\right.
\end{eqnarray*}%
and%
\begin{eqnarray*}
\varphi _{p,2}\left( x\right) &=&x^{1/p}\sqrt{\log _{2}x}\text{ and }\psi
_{p,2}\left( x\right) =\left\vert \frac{x}{\sqrt{\log _{2}x}}\right\vert ^{p}%
\text{,} \\
\text{where }\log _{2}\left( x\right) &=&\left\{ 
\begin{array}{cc}
\log \log \frac{1}{x} & \text{, for }x\leq e^{-e} \\ 
1 & \text{, otherwise}%
\end{array}%
\right.
\end{eqnarray*}
\end{definition}

\begin{remark}
Note that $\varphi _{p,2}\left( \psi _{p,2}\left( s\right) \right) \sim $ $%
\psi _{p,2}\left( \varphi _{p,2}\left( s\right) \right) \sim s$ as $%
s\rightarrow 0$.
\end{remark}

We shall see that (sharp) generalized H\"{o}lder- or variation regularity of
a stochastic process can be shown from the following simple condition.\ It
is not only satisfied by a generic class of Gaussian processes and Gaussian
rough paths (discussed in sections \ref{SS_Gauss}, \ref{SS_GaussRoughPath}
below) but also by Markov processes with uniform (sub)elliptic generator in
divergence form \cite{St88, SaSt91}.

\begin{condition}
\label{ConditionExp}$X$ is a process on $\left[ 0,1\right] $ taking values
in a metric space $\left( E,d\right) $ and there exists a $\eta >0$ s.t.%
\begin{equation}
\sup_{0\leq s<t\leq 1}\mathbb{E}\left( \exp \left( \eta \left[ \frac{d\left(
X_{s},X_{t}\right) }{\left\vert t-s\right\vert ^{1/p}}\right] ^{2}\right)
\right) <\infty ,  \label{EqConditionforGRR}
\end{equation}%
Clearly, this condition guarantees the existence of a continuous version of $%
X$ with which we always work.
\end{condition}

\begin{lemma}
Condition (\ref{EqConditionforGRR}) is equivalent to%
\begin{equation*}
\sup_{0\leq s<t\leq 1}\left\vert \frac{d\left( X_{s},X_{t}\right) }{%
\left\vert t-s\right\vert ^{1/p}}\right\vert _{L^{2q}\left( \mathbb{P}%
\right) }=O\left( \sqrt{q}\right) \text{ as }q\rightarrow \infty .
\end{equation*}
\end{lemma}

\begin{proof}
Left to the reader.
\end{proof}

\subsection{L\'{e}vy's Modulus}

Given $x:\left[ 0,1\right] \rightarrow \left( E,d\right) $ and $\varphi
:[0,\infty )\rightarrow \lbrack 0,\infty )$, strictly increasing, $\varphi
\left( x\right) =0$ iff $x=0$, we set%
\begin{equation*}
\left\vert x\right\vert _{\varphi \text{-H\"{o}l;}\left[ 0,1\right]
}:=\sup_{0\leq s<t\leq 1}\frac{d\left( x_{s},x_{t}\right) }{\varphi \left(
t-s\right) }.
\end{equation*}

\begin{theorem}
\label{ThLevyModulusAS}Let $X$ satisify Condition \ref{ConditionExp}. Assume 
$\varphi \sim \varphi _{p,1}$ near $0+$.Then the random variable 
\begin{equation*}
\left\vert X\left( \omega \right) \right\vert _{\varphi _{p,1}\text{-H\"{o}l;%
}\left[ 0,1\right] }
\end{equation*}%
has a Gauss-tail.
\end{theorem}

\begin{proof}
This is a straight-forward adaption of the case $p=2$ in \cite%
{friz-victoir-05}. We include details as we want to make the point that
there is no obvious extension of these ideas to generalized variation
"norms". The proof is based on the well-known Garsia-Rodemich-Rumsey lemma
with the pair of functions $\psi $,$q$ given by 
\begin{equation*}
\psi \left( x\right) :=e^{\eta x^{2}}-1,q\left( x\right) :=x^{1/p}.
\end{equation*}%
Setting $\zeta \left( x\right) =\int_{0}^{x}u^{1/p-1}\sqrt{\log \left(
1+1/u^{2}\right) }du$ this yields an estimate of form%
\begin{eqnarray*}
d\left( X_{s},X_{t}\right) &\leq &c_{1}\int_{0}^{t-s}u^{1/p-1}\sqrt{\log
\left( 1+4F/u^{2}\right) }\mathrm{d}u \\
&=&c_{1}F^{1/\left( 2p\right) }\zeta \left( \frac{t-s}{\sqrt{F}}\right)
\end{eqnarray*}%
where $\zeta \left( x\right) \sim _{x\rightarrow 0}c_{2}x^{1/p}\sqrt{\log 1/x%
}=c_{2}\varphi _{p,1}\left( x\right) $ for $x$ near $0+$ and $F=F\left(
\omega \right) $ is given by%
\begin{equation}
F=\diint\limits_{\left[ 0,1\right] ^{2}}\psi \left( \frac{d\left(
X_{u},X_{v}\right) }{q\left( \left\vert v-u\right\vert \right) }\right) 
\mathrm{d}u\mathrm{d}v+1  \label{EqGRRZetaNew}
\end{equation}%
By Condition \ref{ConditionExp} and Fubini, $F\in L^{1}\left( \mathbb{P}%
\right) $; adding $1$ guarantees that $F\geq 1$ which will be convenient
below. By Fubini and Conditio we immediately see that $F\in L^{1}$
(especially finite a.s.). An elementary computation reveals 
\begin{equation}
\exists c_{2}:\forall x,y\in \left[ 0,1\right] :\zeta \left( xy\right) \leq
c_{2}\zeta \left( x\right) \zeta \left( y\right) .  \label{EqZetaSubm}
\end{equation}%
Combined with (\ref{EqGRRZetaNew}) we see that 
\begin{equation}
d\left( X_{s},X_{t}\right) \leq c_{3}F^{1/\left( 2p\right) }\zeta \left( 1/%
\sqrt{F}\right) \zeta \left( t-s\right) .  \label{EqModulusWithZeta}
\end{equation}%
It remains to see that $M\left( \omega \right) =M:=F^{1/\left( 2p\right)
}\zeta \left( 1/\sqrt{F}\right) $ has a Gauss tail. After a change of
variables ($\tilde{u}=F^{1/2}u$) $M=\int_{0}^{1}\sqrt{\log \left(
1+F/u^{2}\right) }\mathrm{d}u$ and Jensen's inequality gives%
\begin{eqnarray}
\mathbb{E}\left[ \exp \left( \lambda M^{2}\right) \right] &\leq &\mathbb{E}%
\left[ \int_{0}^{1}\exp \left[ \lambda \log \left( 1+F/u^{2}\right) \right] 
\mathrm{d}u\right]  \notag \\
&\leq &\mathbb{E}\left[ \int_{0}^{1}\exp \left[ \lambda \log \left(
2F/u^{2}\right) \right] \mathrm{d}u\right]  \notag \\
&=&\left( 2F\right) ^{\lambda }\mathbb{E}\left[ \int_{0}^{1}1/u^{2\lambda }%
\mathrm{d}u\right] .  \label{EqMGaussTail}
\end{eqnarray}%
It now suffices to choose $\lambda \in \left( 0,1/2\right) $, so that the
deterministic integral is finite, and to observe that $1\leq F^{\lambda
}\leq F\in L^{1}\left( \mathbb{P}\right) $. To see that Gauss tail of $%
\varphi _{p,1}$-H\"{o}lder "norm", we split up the $\sup $. For
deterministic $\delta $, small enough, we have%
\begin{eqnarray*}
\sup_{0\leq s<t\leq 1}\frac{d\left( X_{s},X_{t}\right) }{\varphi
_{p,1}\left( \left\vert t-s\right\vert \right) } &\leq &\sup_{\substack{ %
0\leq s<t\leq 1  \\ \left\vert t-s\right\vert \leq \delta }}\frac{d\left(
X_{s},X_{t}\right) }{\zeta \left( \left\vert t-s\right\vert \right) }\frac{%
\zeta \left( t-s\right) }{\varphi _{p,1}\left( \left\vert t-s\right\vert
\right) }+\sup_{\substack{ 0\leq s<t\leq 1  \\ \left\vert t-s\right\vert
>\delta }}\frac{d\left( X_{s},X_{t}\right) }{\varphi _{p,1}\left( \left\vert
t-s\right\vert \right) } \\
&\leq &c_{4}M+c_{5}\left\vert X\right\vert _{0,\left[ 0,1\right] }.
\end{eqnarray*}%
Using Gaussian integrability of $M$ and $\left\vert X\right\vert _{0,\left[
0,1\right] }$, cf. Theorem \ref{ThExpTail}, we see that $\left\vert
X\right\vert _{\varphi _{p,1}\text{-H\"{o}l},\left[ 0,1\right] }$ has a
Gauss tail. The same argument works for any other $\varphi $-modulus for
which $\varphi \sim \varphi _{p,1}$ near $0+$. th and in combination with $%
\left( \ref{EqModulusWithZeta}\right) $ and $\left( \ref{EqMGaussTail}%
\right) $ the last expression is the sum of two r.v. with a Gauss tail. It
is easy to see that the sum of two r.v. with a Gauss tail has again a Gauss
tail. The proof is finished by the observation that for every $\varphi $
with $\varphi \left( x\right) \sim \varphi _{p,1}\left( x\right) $ for $%
x\rightarrow 0$ the same argument works.
\end{proof}

\begin{theorem}
\label{ThExpTail}Let $X$ satisify Condition \ref{ConditionExp}. Then there
exists a constant $C$ such that for all $a<b$ in $\left[ 0,1\right] $ we
have 
\begin{equation*}
\mathbb{P}\left[ \left\vert X\right\vert _{0,\left[ a,b\right] }>x\right]
\leq C\exp \left( -\frac{1}{C}\left( \frac{x}{\left\vert b-a\right\vert
^{1/p}}\right) ^{2}\right)
\end{equation*}%
where $\left\vert X\right\vert _{0,\left[ a,b\right] }:=\sup_{a\leq s<t\leq
b}d\left( X_{s},X_{t}\right) .$
\end{theorem}

\begin{proof}
\begin{equation*}
\sup_{a\leq s<t\leq b}\mathbb{E}\exp \left( \eta \left( \frac{d\left(
X_{s},X_{t}\right) }{\left\vert t-s\right\vert ^{1/p}}\right) ^{2}\right)
<\infty
\end{equation*}%
implies, setting $Z_{t}=X_{a+t\left( b-a\right) },$%
\begin{equation*}
\sup_{s,t\in \left[ 0,1\right] }\mathbb{E}\exp \left( \eta \left( \frac{%
d\left( Z_{s},Z_{t}\right) }{\left\vert b-a\right\vert ^{1/p}\left\vert
t-s\right\vert ^{1/p}}\right) ^{2}\right) <\infty .
\end{equation*}
We have $\left( b-a\right) ^{\alpha }\left\vert X\right\vert _{\alpha \text{%
-H\"{o}l;}\left[ a,b\right] }=\left\vert Z\right\vert _{\alpha \text{-H\"{o}%
l;}\left[ 0,1\right] }$ and by Garsia-Rodemich-Rumsey, for any $0\leq \alpha
<1/p,$ 
\begin{equation*}
\exists \tilde{\eta}>0:\mathbb{E}\exp \left( \tilde{\eta}\left( \frac{%
\left\vert Z\right\vert _{\alpha \text{-H\"{o}l;}\left[ 0,1\right] }}{%
\left\vert b-a\right\vert ^{1/p}}\right) ^{2}\right) <\infty .
\end{equation*}%
It now suffices to take $\alpha =0$ and use Markov's inequality.
\end{proof}

\subsection{Taylor's Variation}

\begin{theorem}
\label{TaylorVar}Let $X$ satisify Condition \ref{ConditionExp}. Then with
probability $1$,%
\begin{equation}
V_{\psi _{p,2}}\left( X\right) :=\sup_{D\subset \left[ 0,1\right]
}\sum_{i:t_{i}\in D}\psi _{p,2}\left( d\left( X_{t_{i}},X_{t_{i+1}}\right)
\right) <\infty \text{.}  \label{EqPsiVarfinite}
\end{equation}%
In the notation of appendix \ref{AppendixVariationNorms} this is equivalent
to%
\begin{equation*}
\left\vert X\right\vert _{\psi _{p,2}\text{-var;}\left[ 0,1\right] }<\infty 
\text{ a.s.}
\end{equation*}
\end{theorem}

Let us remark that (\ref{EqPsiVarfinite}) holds for any function $\psi $ (in
a reasonable class, cf. appendix A) for which%
\begin{equation*}
\lim \sup_{s\rightarrow 0}\psi \left( s\right) /\psi _{p,2}\left( s\right)
<\infty ;
\end{equation*}%
this follows readily from (\ref{EqPsiVarfinite})$\Leftrightarrow $(\ref%
{EqLimDeltaPsiVar}). Good examples include $s\mapsto s^{p+\varepsilon }$ and 
$s\mapsto \psi _{p,1}\left( s\right) $. It should be emphasized that the
latter statement $\left\vert X\right\vert _{\psi _{p,1}\text{-var;}\left[ 0,1%
\right] }<\infty $ a.s.\ is a trivial consequence of $\varphi _{p,1}$-H\"{o}%
lder regularity (cf.\ Theorem \ref{ThLevyModulusAS}). However, examples show
that finite $\psi _{p,2}$-variation can hold without having finite $\varphi
_{p,2}$-H\"{o}lder modulus (e.g.\ Brownian motion with $p=2$).

\begin{lemma}
Consider a sequence of positive real numbers $\left( h_{n}\right) \downarrow
0$ such that\footnote{%
This rules out $e^{-n^{2}}$ for instance.}%
\begin{equation*}
C:=2\sup_{n}h_{n-1}/h_{n}<\infty
\end{equation*}%
and define for each $n\in \left\{ 1,2,\dots \right\} $ a family of intervals%
\begin{equation*}
\mathcal{J}_{n}:=\left\{ J_{n,i}:=\left[ \frac{i}{2}h_{n},\left( \frac{i}{2}%
+1\right) h_{n}\right] :i\in \left\{ 1,\dots ,\left[ 2/h_{n}\right]
+1\right\} \right\}
\end{equation*}%
Then there exists a $\delta >0$ only depending on $\left( h_{n}\right) $
such that any interval $\left( s,t\right) \subset \left[ 0,1\right] $ with $%
\left\vert t-s\right\vert <\delta $ is well approximated by some interval $%
J_{n,i}$ in the sense that the following conditions are satisfied,%
\begin{equation*}
\left( t-s\right) \subset J_{n,i}\text{ and \ }\left\vert J_{n,i}\right\vert
=h_{n}\leq C\left\vert t-s\right\vert .\text{ }
\end{equation*}%
Furthermore, for fixed $\left( h_{n}\right) $ the choice of $n$ depends only
on $\left( t-s\right) $ and $n\uparrow \infty $ as $\left( t-s\right)
\downarrow 0$.
\end{lemma}

\begin{proof}
We choose the largest $n$ which still satisfies the second condition, that is%
\begin{equation}
h_{n}\leq C\left( t-s\right) <h_{n-1}.  \label{LemmaTaylorVarChoiceOfN}
\end{equation}%
We then choose the largest possible index $i$ so that $J_{n,i}$ satisfies
the first condition, that is%
\begin{equation*}
\frac{i}{2}h_{n}<s<\frac{i+1}{2}h_{n},
\end{equation*}%
and note that $\left( s-\frac{i}{2}h_{n}\right) <h_{n}/2$. To see that $%
\left( s,t\right) $ is indeed contained in $J_{n,i}$ it is enough to check
that%
\begin{equation*}
h_{n}/2+\left( t-s\right) \leq \left\vert J_{n,i}\right\vert =h_{n}\text{ or
equivalently \ }\frac{t-s}{h_{n}}\leq \frac{1}{2}\text{.}
\end{equation*}%
But this is true by (\ref{LemmaTaylorVarChoiceOfN}) and definition of $C$
since 
\begin{equation*}
\frac{t-s}{h_{n}}=\frac{t-s}{h_{n-1}}\frac{h_{n-1}}{h_{n}}\leq \frac{t-s}{%
C\left( t-s\right) }\sup_{n}\frac{h_{n-1}}{h_{n}}=\frac{1}{2}.
\end{equation*}
\end{proof}

\begin{proof}[Proof of Theorem \protect\ref{TaylorVar}]
As in \cite{Taylor72} it is enough to show that 
\begin{equation}
\lim_{\delta \rightarrow 0}\sup_{D\in \mathcal{D}\left( \delta \right)
}\sum_{i:t_{i}\in D}\psi _{p,2}\left( d\left( X_{t_{i}},X_{t_{i+1}}\right)
\right) <\infty \text{ a.s.}  \label{EqLimDeltaPsiVar}
\end{equation}%
For a given $D$ call an interval $\left( t_{i-1},t_{i}\right) $ in the
dissection $D$ \textit{good} if $\psi _{p,2}\left( d\left(
X_{t_{i}},X_{t_{i+1}}\right) \right) <c_{1}\left( t_{i}-t_{i-1}\right) $ for
some deterministic $c_{1}$ to be determined by equation (\ref%
{EqChooseConstant}) below. Call this interval \textit{bad} otherwise.
Clearly,%
\begin{eqnarray}
\sum_{i}\psi _{p,2}\left( \left\vert X\right\vert \right) &=&\sum_{\text{%
good intervals }\left( t_{i-1},t_{i}\right) }+\sum_{\text{bad intervals }%
\left( t_{i-1},t_{i}\right) }  \notag \\
&\leq &c_{1}+\sum_{\text{bad intervals }\left( t_{i-1},t_{i}\right) }
\label{DetBoundForVarPsiwithMeshto0}
\end{eqnarray}%
and we only need to deal with bad intervals. We will see that, provided $%
\left\vert D\right\vert <\delta $ is small enough, the sum over the bad
intervals can be controlled. Let $\left( s,t\right) $ be a bad interval in $%
D $ with $\left\vert t-s\right\vert <\delta _{1}$ where $\delta _{1}$ is the
constant whose existence is guaranteed by the previous lemma. The same
lemma, applied with $h_{n}=e^{-n}$ and $C=2e$, implies that we can find $i$
and $n$ such that 
\begin{equation*}
\left( s,t\right) \subset J_{n,i}=\left[ \frac{i}{2}h_{n},\left( \frac{i}{2}%
+1\right) h_{n}\right]
\end{equation*}%
and $h_{n}=\left\vert J_{n,i}\right\vert <2e\left( t-s\right) $. In
particular,%
\begin{equation}
\psi _{p,2}\left( \left\vert X\right\vert _{0,J_{n,i}}\right) \geq \psi
_{p,2}\left( d\left( X_{s},X_{t}\right) \right) \geq c_{1}\left( t-s\right)
\geq c_{2}h_{n}  \label{TaylorBefore4point7}
\end{equation}%
where we set $c_{2}=c_{1}/\left( 2e\right) $ and used that $\left(
s,t\right) $ is bad. Recalling $\varphi _{p,2}\left( \psi _{p,2}\left(
s\right) \right) \sim s$ as $s\rightarrow 0$ and $\varphi _{p,2}\left( \psi
_{p,2}\left( s\right) \right) =s$ for $s\geq 1$ we obviously have $\varphi
_{p,2}\left( \psi _{p,2}\left( s\right) \right) \leq c_{3}s$. \ Hence, using
in particular Theorem \ref{ThExpTail}, and writing $c_{4}$ for the constant
whose existence it guarantees,%
\begin{eqnarray}
&&\mathbb{P}\left[ \psi _{p,2}\left( \left\vert X\right\vert
_{0,J_{n,i}}\right) >c_{2}h_{n}\right]  \notag \\
&\leq &\mathbb{P}\left[ c_{3}\left\vert X\right\vert _{0,J_{n,i}}>\varphi
_{p,2}\left( c_{2}h_{n}\right) \right]  \notag \\
&\leq &c_{4}\exp \left[ -\frac{1}{c_{4}}\left( \frac{\varphi _{p,2}\left(
c_{2}h_{n}\right) }{c_{3}h_{n}^{1/p}}\right) ^{2}\right]  \notag \\
&=&c_{4}\exp \left[ -\frac{1}{c_{4}}\left( \frac{c_{2}^{1/p}}{c_{3}}\right)
^{2}\log _{2}\left( c_{2}h_{n}\right) \right] .  \label{EqTailPsiOfX}
\end{eqnarray}%
We now choose $c_{1}$ such that%
\begin{equation}
\frac{1}{c_{4}}\left( \frac{c_{1}^{1/p}}{c_{3}\left( 2e\right) ^{1/p}}%
\right) ^{2}=\frac{1}{c_{4}}\left( \frac{c_{2}^{1/p}}{c_{3}}\right) ^{2}=5p
\label{EqChooseConstant}
\end{equation}%
Note that for $n$ greater than some $n_{1}$ large enough, $\log _{2}\left(
c_{2}h_{n}\right) =\log \left( -\log \left( c_{2}h_{n}\right) \right) $, and
(\ref{EqTailPsiOfX}) reads%
\begin{equation*}
c_{4}\frac{1}{\left( -\log \left( c_{2}h_{n}\right) \right) ^{5p}}\leq c_{4}%
\frac{2^{5p}}{n^{5p}}\equiv c_{5}\frac{1}{n^{5p}}
\end{equation*}%
where the estimate holds true provided $n\geq n_{2}$ large enough so that 
\begin{equation}
-\log \left( c_{2}h_{n}\right) =-\log c_{2}+n\geq n/2.
\label{EqLog2Estimate}
\end{equation}%
We established that%
\begin{equation*}
\mathbb{P}\left[ \psi _{p,2}\left( \left\vert X\right\vert
_{0,J_{n,i}}\right) >c_{2}h_{n}\right] \leq c_{5}n^{-5p}\text{ for }n\geq
n_{1}\vee n_{2}.
\end{equation*}%
If $Z_{n}=Z_{n}\left( \omega \right) $ denotes the number of intervals in $%
\mathcal{J}_{n}$ which satisfy%
\begin{equation}
\psi _{p,2}\left( \left\vert X\right\vert _{0,J_{n,i}}\right) >c_{2}h_{n}
\label{Taylor4point7}
\end{equation}%
Since $Z_{n}$ is the sum (over $i=1,\dots ,\left[ 2/h_{n}\right] +1$) of all
indicator functions of the events $\left\{ \psi _{p,2}\left( \left\vert
X\right\vert _{0,J_{n,i}}\right) >c_{2}h_{n}\right\} ,$%
\begin{eqnarray*}
\mathbb{E}\left( Z_{n}\right) &\leq &\left\vert \mathcal{J}_{n}\right\vert
\times c_{5}n^{-5p} \\
&\leq &c_{5}\left( 2/h_{n}+1\right) n^{-5p} \\
&=&c_{6}h_{n}^{-1}n^{-5p}
\end{eqnarray*}%
Introduce the event $A_{n}=\left\{ Z_{n}>n^{-2p}h_{n}^{-1}\right\} $. Then%
\begin{equation*}
\mathbb{P}\left( A_{n}\right) \leq n^{2p}h_{n}\mathbb{E}\left( Z_{n}\right)
=c_{6}n^{-3p}
\end{equation*}%
and $\sum_{n}\mathbb{P}\left( A_{n}\right) <\infty $, after all we have $%
p\geq 1$ fixed. The Borel-Cantelli lemma now implies that $\mathbb{P}\left(
A_{n}\text{ infinitely often}\right) =0$. Equivalently, with probability $1$
there exists $N\left( \omega \right) $ such that%
\begin{equation*}
Z_{n}\leq n^{-2p}h_{n}^{-1}\text{ for all }n\geq N\left( \omega \right) .
\end{equation*}%
Using a.s.\ finiteness of the L\'{e}vy's modulus "norm", theorem \ref%
{ThLevyModulusAS}, there exists $C_{6}\left( \omega \right) $, finite almost
surely, so that for any $i\in \left\{ 1,\dots ,\left[ 2/h_{n}\right]
+1\right\} $,%
\begin{equation*}
\left\vert X\right\vert _{0,J_{n,i}}<C_{6}\,h_{n}^{1/p}\sqrt{\log _{1}h_{n}}
\end{equation*}%
From our definition of $\psi _{p,2}$ we have $\psi _{p,2}\left( s\right)
\leq s^{p}$ for all $s$ and so 
\begin{equation*}
\psi _{p,2}\left( \left\vert X\right\vert _{0,J_{n,i}}\right) \leq
C_{6}^{p}h_{n}\left( \log _{1}h_{n}\right) ^{p/2}.
\end{equation*}%
Now, for $n\geq N\left( \omega \right) $, the sum of $\psi _{p,2}\left(
\left\vert X\right\vert _{0,J_{n,i}}\right) $ over all intervals $J_{n,i}\in 
\mathcal{J}_{n}$ which satisfy (\ref{Taylor4point7}) is at most%
\begin{equation}
Z_{n}\times C_{6}^{p}h_{n}\left( \log _{1}h_{n}\right) ^{p/2}\leq
C_{6}^{p}n^{-2p}\left( \log _{1}h_{n}\right) ^{p/2}=C_{6}^{p}n^{-3p/2}.
\label{TaylorEstimateForFinalStep}
\end{equation}%
As remarked in (\ref{TaylorBefore4point7}) every bad interval $\left(
s,t\right) $ of lenght smaller than $\delta _{1}$ is contained in some $%
J_{n,i}\in \mathcal{J}_{n}$ and such that $\psi _{p,2}\left( \left\vert
X\right\vert _{0,J_{n,i}}\right) \geq c_{2}h_{n}.$Let $\delta =\delta \left(
\omega \right) \in \left( 0,\delta _{1}\right) $ be small enough such that
the (in the sense of lemma above) $n=n\left( \delta \right) >N\left( \omega
\right) $ and so we are only dealing with intervals $J_{n,i}$ to which our
estimates apply. Then for any partition $D$ with $\left\vert D\right\vert
<\delta \left( \omega \right) $, 
\begin{eqnarray*}
&&\sum_{\text{bad intervals }\left( t_{i-1},t_{i}\right) \in D}\psi
_{p,2}\left( d\left( X_{t_{i}},X_{t_{i+1}}\right) \right) \\
&\leq &\sum_{n=m\left( \delta \left( \omega \right) \right) }^{\infty
}\sum_{ _{\substack{ J_{n,i}\in \mathcal{J}_{n}\text{ for which}  \\ \text{(%
\ref{Taylor4point7}) holds}}}}\psi _{p,2}\left( \left\vert X\right\vert
_{0,J_{n,i}}\right) \\
&\leq &\sum_{n=m\left( \delta \left( \omega \right) \right) }^{\infty
}C_{6}^{p}n^{-p3/2}\text{ \ \ thanks to (\ref{TaylorEstimateForFinalStep}).}
\end{eqnarray*}%
and this sum is finite almost surely as required. (The last step actually
shows that we get a deterministic upper bound in (\ref{EqLimDeltaPsiVar})
but this is irrelevant for our purposes.)
\end{proof}

\subsection{Law of Iterated Logarithm}

A law of iterated logarithm also holds in the generality of the present
setup of continuous processes on $\left[ 0,1\right] $ with values in some
metric space.

\begin{proposition}
\label{LIL}Let $X$ satisify Condition \ref{ConditionExp}.Then there exists a
constant $C<\infty $ s.t. 
\begin{equation*}
\lim \sup_{h\downarrow 0}\frac{\left\vert X\right\vert _{0;\left[ 0,h\right]
}}{\varphi _{p,2}\left( h\right) }\leq C\text{ a.s.}
\end{equation*}
\end{proposition}

\begin{proof}
The idea is to scale by a geometric sequence. Although not strictly
necessary for the conclusion, we show that $C=\sqrt{c_{1}}$ where $c_{1}$ is
the constant whose existence is guaranteed by Theorem \ref{ThExpTail}. To
this end, fix $\varepsilon >0,$ $q\in \left( 0,1\right) $ and set $c_{2}=%
\sqrt{\left( 1+\varepsilon \right) c_{1}}$. Let%
\begin{equation*}
A_{n}=\left\{ \left\vert X\right\vert _{0;\left[ 0,q^{n}\right] }\geq
c_{2}\varphi _{p,2}\left( q^{n}\right) \right\} .
\end{equation*}%
From Theorem \ref{ThExpTail} we see that for $n$ large enough%
\begin{eqnarray*}
\mathbb{P}\left( A_{n}\right) &=&\mathbb{P}\left[ \left\vert X\right\vert
_{0;\left[ 0,q^{n}\right] }\geq c_{2}\varphi _{p,2}\left( q^{n}\right) %
\right] \\
&\leq &c_{1}\exp \left( -\frac{1}{c_{1}}\left( \frac{c_{2}\varphi
_{p,2}\left( q^{n}\right) }{q^{n/p}}\right) ^{2}\right) \\
&=&c_{1}\exp \left( -\frac{1}{c_{1}}c_{2}^{2}\log _{2}q^{n}\right) \\
&=&c_{1}\left( -n\log q\right) ^{-c_{2}^{2}/c_{1}}
\end{eqnarray*}%
This is summable in $n$ and hence, by the Borel-Cantelli lemma, we get that
only finitely many of these events occur. It then follows easily that for
all $n\geq n_{0}\left( \varepsilon ,\omega \right) $ and $h$ small enough%
\begin{equation*}
q^{n+1}\leq h<q^{n}
\end{equation*}%
and so, since $\varphi _{p,2}\left( h\right) /h$ is decreasing, $\varphi
_{p,2}\left( q^{n}\right) /\varphi _{p,2}\left( q^{n+1}\right) \leq q^{-1}$,
and then 
\begin{equation*}
\frac{\left\vert X\right\vert _{0;\left[ 0,h\right] }}{\varphi _{p,2}\left(
h\right) }\leq \frac{\varphi _{p,2}\left( q^{n}\right) }{\varphi
_{p,2}\left( q^{n+1}\right) }\frac{\left\vert X\right\vert _{0;\left[ 0,q^{n}%
\right] }}{\varphi _{p,2}\left( q^{n}\right) }\frac{\varphi _{p,2}\left(
q^{n+1}\right) }{\varphi _{p,2}\left( h\right) }\leq q^{-1}\sqrt{\left(
1+\varepsilon \right) c_{1}}.
\end{equation*}%
We now pass to $\lim \sup_{h\rightarrow 0}$, followed by $q\uparrow 1$ and $%
\varepsilon \downarrow 0$. This finishes the proof.
\end{proof}

\section{Integrability of $\protect\psi _{p,2}$-variation norm\label%
{SectionIntegrability}}

We have seen that a (continuous) process $X:\left[ 0,1\right] \rightarrow
\left( E,d\right) $ which satisifies Condition \ref{ConditionExp} has a.s.\
finite $\psi _{p,2}$-variation. The aim of this section is to show that, for
large classes of Gaussian processes and Gaussian rough paths, the $\psi
_{p,2}$-variation "norm" (cf.\ appendix) enjoys Gaussian integrability. Due
to Theorem \ref{ThGeneralizedFernique} this reduces to check if the
assumptions (\ref{EqfFinite}) and (\ref{Control_H_translate}) of Theorem \ref%
{ThGeneralizedFernique} are satisfied.

\subsection{Gaussian paths\label{SS_Gauss}}

Assume that $X$ is a centered (continuous) Gaussian process so that%
\begin{equation*}
\sup_{D}\sum_{i:\left( t_{i}\right) \in D}\left\vert \mathbb{E}\left( \left(
X_{t_{i},t_{i+1}}\right) ^{2}\right) \right\vert ^{\rho }<\infty
\end{equation*}%
with $\sup $ taken over all dissections $D$ of $\left[ 0,1\right] $. This
condition appears in \cite{JainMonrad83} for instance. (For orientation, it
holds for Brownian motion with $\rho =1$ and fractional Brownian motion with 
$\,1/\rho =1/\left( 2H\right) $). Then a deterministic time-change of $X$,
say $Z$, satisfies%
\begin{equation*}
\sup_{0\leq s<t\leq 1}\frac{\mathbb{E}\left( \left\vert Z_{s,t}\right\vert
^{2}\right) }{\left\vert t-s\right\vert ^{1/\rho }}<\infty .
\end{equation*}%
With Gaussian integrability properties, $\mathbb{E}\left( \left\vert
Z_{s,t}\right\vert ^{p}\right) ^{1/p}\sim q^{1/2}\mathbb{E}(\left\vert
Z_{s,t}\right\vert ^{2})^{/2}$ this readily implies that $Z$ satisfies
condition \ref{ConditionExp}. Using the invariance of (generalized)
variation norms under reparametrization and theorem \ref{TaylorVar} we
conclude that the sample paths of $X$ are almost surely of finite $\psi
_{p,2}$-variation for $p=2\rho $. It is clear from the remark following
theorem \ref{TaylorVar} that one also has finite $V_{\psi ;\left[ 0,1\right]
}\left( x\right) <\infty $ for any other $\psi $-function such that $\lim
\sup_{s\rightarrow 0}\psi \left( s\right) /\psi _{p,2}\left( s\right)
<\infty $. If in addition to the standing assumpitions (cf. appendix A) one
has convexity of $\psi $, then $x\mapsto \left\vert x\right\vert _{\psi 
\text{-var;}\left[ 0,1\right] }$ gives rise to a Banach-norm $x\mapsto
\left\vert x\right\vert _{\psi \text{-var;}\left[ 0,T\right] }$ and a Gauss
tail follows from Fernique's classical result (or Theorem \ref%
{ThGeneralizedFernique} applied to $f=\left\vert .\right\vert $).

\subsection{Gaussian Rough paths\label{SS_GaussRoughPath}}

Fernique's classical estimates are not applicable to Gaussian rough paths
since the homogenuous norms involve the path level \textit{and }L\'{e}vy
area (which is not Gaussian). However, Gauss tail estimates for the $\psi
_{2,p}$-variation norm (and any other $\psi $ such that $\psi \equiv \psi
_{2,p}$ near $0+$) do follow from our generalized Fernique theorem. We
emphasize (again) that integrability properties of Wiener-It\^{o} chaos
(even Banach space valued) will not be sufficient for these estimates since
the homogenous norms we are dealing with are fundamentally non-linear.

We shall need the result below.

\begin{theorem}[\protect\cite{friz-victoir-07-DEdrivenGaussian_I}]
\label{Th2dimVarCMspace}Let $X$ be a continuous, centered Gaussian process
on $\left[ 0,1\right] $ with covariance $R\left( s,t\right) =\mathbb{E}%
\left( X_{s}X_{t}\right) $ such that 
\begin{equation*}
\left\vert R\right\vert _{\rho \text{-var,}\left[ 0,1\right] ^{2}}:=\sup_{D,%
\tilde{D}\subset \left[ 0,1\right] }\left( \sum_{i,j:t_{i}\in D,\tilde{t}%
_{j}\in \tilde{D}}\left\vert \mathbb{E}\left[ X_{t_{i},t_{i+1}}X_{\tilde{t}%
_{j},\tilde{t}_{j+1}}\right] \right\vert ^{\rho }\right) ^{1/\rho }<\infty
\end{equation*}%
finite. Then, if $\mathcal{H}$ denotes the Cameron-Martin space associated
to $X$, we have the continous embedding 
\begin{equation*}
\mathcal{H\hookrightarrow }\text{ }C^{\rho \text{-var}}\left( \left[ 0,1%
\right] ,\mathbb{R}^{d}\right) .
\end{equation*}%
More precisely, for all $h\in \mathcal{H}$ and all $0\leq s<t\leq 1$, 
\begin{equation*}
\left\vert h\right\vert _{\rho \text{-var,}\left[ s,t\right] }\leq \sqrt{%
\left\langle h,h\right\rangle _{\mathcal{H}}}\sqrt{R_{\rho \text{-var,}\left[
s,t\right] ^{2}}}.
\end{equation*}
\end{theorem}

We recall that a Lipschitz continuous path in $\mathbb{R}^{d}$, by simple
computation of its area integral, lifts to a path in $G^{2}\left( \mathbb{R}%
^{d}\right) $, the step-$2$ nilpotent group with $d$ generators equipped. We
equip $G^{2}\left( \mathbb{R}^{d}\right) $ with the Carnot-Caratheodory
metric $d$; this is natural, for instance the lifted path is then Lipschitz
with respect to $d$. H\"{o}lder and variation regularity of paths $\left[ 0,1%
\right] \rightarrow \left( G^{2}\left( \mathbb{R}^{d}\right) ,d\right) $ are
then special cases of the general discussion we had previously. We denote
the resulting pathspace "norms" by $\left\Vert \cdot \right\Vert _{\varphi 
\text{-H\"{o}l}}$ and $\left\Vert \cdot \right\Vert _{\psi \text{-var}}~$%
etc.\ to distinguish from the general case.

\begin{theorem}
\label{ThGausstailsXPsvar}Let $X$ be a centered, continuous Gaussian process
in $\mathbb{R}^{d}$ on $\left[ 0,1\right] $ with independent components. If
the covariation of $X$ is of finite $\rho $-variation for $\rho <3/2$ then
there exists a lift to a Gaussian rough path $\mathbf{X}\in C_{0}\left( %
\left[ 0,1\right] ,G^{2}\left( \mathbb{R}^{d}\right) \right) $ of finite
homogenuous $\left( 2\rho +\varepsilon \right) $-variation, $\varepsilon >0$
and $\left\Vert \mathbf{X}\right\Vert _{\psi _{p,2}\text{-var;}\left[ 0,1%
\right] }$ has a Gauss tail for $p=2\rho $.\ More precisely, there exists $%
\eta >0$ such that%
\begin{equation*}
\int \exp \left( \eta \left\Vert \mathbf{X}\left( \omega \right) \right\Vert
_{\psi _{p,2}\text{-var;}\left[ 0,1\right] }^{2}\right) d\mu \left( \omega
\right) <\infty .
\end{equation*}
\end{theorem}

\begin{proof}
The first part of the statement (existence of lift) is proven in \cite%
{friz-victoir-07-DEdrivenGaussian_I}. To show the Gauss tail we apply
Theorem \ref{ThGeneralizedFernique} with $f=\left\Vert \cdot \right\Vert
_{\psi _{p,2}\text{-var;}\left[ 0,1\right] }\circ S_{2}$ where $%
S_{2}:X\left( \omega \right) =\omega \mapsto \mathbf{X}\left( \omega \right) 
$ denotes the lift constructed in \cite{friz-victoir-07-DEdrivenGaussian_I}%
\footnote{%
A measurable map from%
\begin{equation*}
C_{0}\left( \left[ 0,1\right] ,\mathbb{R}^{d}\right) \rightarrow C_{0}\left( %
\left[ 0,1\right] ,G^{2}\left( \mathbb{R}^{d}\right) \right) .
\end{equation*}%
}, setting $p=2\rho $ we need to check (i) $\left\Vert \mathbf{X}\right\Vert
_{\psi _{p,2}\text{-var;}\left[ 0,1\right] }<\infty $ a.s.\ and (ii)%
\begin{equation}
\left\Vert \mathbf{X}\right\Vert _{\psi _{p,2}\text{-var;}\left[ 0,1\right]
}\leq c\left( \left\Vert T_{-h}\mathbf{X}\right\Vert _{\psi _{p,2}\text{-var;%
}\left[ 0,1\right] }+\sigma \left\vert h\right\vert _{\mathcal{H}}\right) .
\label{cFromTranslate}
\end{equation}%
Ad (i): We reparametrize\ (using the inverse of $t\mapsto \left\vert
R\right\vert _{\rho \text{-var;}\left[ 0,t\right] ^{2}}^{\rho }$) and obtain
a continuous $G^{2}\left( \mathbb{R}^{d}\right) $-valued process $\mathbf{Z}$
which satisfies (\cite[Thm 35, equation (16)]%
{friz-victoir-07-DEdrivenGaussian_I}%
\begin{equation*}
\left\vert \left\Vert \mathbf{Z}_{s,t}\right\Vert \right\vert _{L^{q}\left( 
\mathbb{P}\right) }\leq C\sqrt{q}\left\vert t-s\right\vert ^{\frac{1}{2\rho }%
}.
\end{equation*}%
It now follows from theorem \ref{TaylorVar} that $\mathbf{Z}$ has a.s.\
finite $\psi _{p,2}$-variation for $p=2\rho $ and thanks to invariance of
generalized variation the same is true for $\mathbf{X}$.

Ad (ii). This is precisely theorem \ref{ThTranslationPhiVarAndRhoVar} below.
We see in particular that any $\eta <1/\left( 2c^{2}\sigma ^{2}\right) $
where $\sigma $ was defined in (\ref{DefSigmaAbstractWienerSpace}) and $c$
is the constant which appears in (\ref{cFromTranslate}) below. This finishes
the proof.
\end{proof}

\begin{remark}
This theorem applies in particular to Brownian motion and fractional
Brownian motion and their enhancement to rough paths (e.g.\ \cite%
{coutin-qian-02}), at least for Hurst parameter $H>1/3$. (A Besov-variation
embedding theorem \cite{friz-victoir-05-JFA} shows that $H>1/4$ is
possible.) Using the well-known results of Taylor \cite{Taylor72},
Kawada-Kono \cite{KawadaKono73} combined with the trivial $\left\vert
X\right\vert _{\psi _{p,2}\text{-var}}\leq \left\Vert \mathbf{X}\right\Vert
_{\psi _{p,2}\text{-var}}$ we see from these examples that one cannot hope
for stronger statements of this type.
\end{remark}

\section{Applications\label{sectionApplications}}

\subsection{Regularity and Integrability of L\'{e}vy's Area}

We now focus on the L\'{e}vy area $A$ of $d$-dimensional Brownian motion $B$%
, defined by 
\begin{equation*}
A_{s,t}=\frac{1}{2}\left( \int_{s}^{t}\left( B_{u}-B_{s}\right) \otimes 
\mathrm{d}B_{u}-\int_{s}^{t}\left( B_{u}-B_{s}\right) \otimes \mathrm{d}%
B_{u}\right) .
\end{equation*}%
We already know that the Brownian motion enhanced with L\'{e}vy's area gives
rise to a rough path of $\psi _{2,2}$-variation and the optimality of this
is a consequence of the optimality of $\psi _{2,2}$-variation of the
underlying Brownian motion alone. We now show that implicit variation
regularity of L\'{e}vy's area is optimal in its own right.\textit{\ }

\begin{proposition}[LIL for the Area process]
There exists a positive constant $C$, such that 
\begin{equation*}
\limsup_{h\rightarrow 0}\frac{\sqrt{\left\vert A_{0,h}\right\vert }}{\varphi
_{2,2}\left( h\right) }\geq C\text{ a.s.}
\end{equation*}%
As a consequence, for every fixed $t\in \lbrack 0,1)$, 
\begin{equation}
\limsup_{h\rightarrow 0}\frac{\sqrt{\left\vert A_{t,t+h}\right\vert }}{%
\varphi _{2,2}\left( h\right) }\geq C\text{ a.s.}  \label{CorLILAreaNew}
\end{equation}
\end{proposition}

\begin{proof}
See \cite[Theorem $2.15$]{Neuenschwander96:ProbaonHeisenbergGroup}. The
consequence follows from $A_{0,\cdot }\overset{law}{=}A_{t,t+\cdot }$ for
fixed $t$.
\end{proof}

\begin{theorem}
\label{LowerBoundLevyAreaPsiVar}There exists a positive constant $C$ such
that 
\begin{equation*}
\lim_{\delta \downarrow 0}\sup_{\left\vert D\right\vert <\delta
}\sum_{i:t_{i}\in D}\psi _{2,2}\left( \left\vert
A_{t_{i},t_{i+1}}\right\vert ^{1/2}\right) \geq C\text{ a.s.}
\end{equation*}
\end{theorem}

\begin{proof}
For brevity, we set $\psi :=\psi _{2,2}$. It is convenient to define $%
V_{\psi \circ \sqrt{.}}^{D}\left( A\right) :=\sum_{i:t_{i}\in D}\psi \circ 
\sqrt{.}\left( \left\vert A_{t_{i},t_{i+1}}\right\vert \right) $. Then the
proof works along the lines of \cite{Taylor72}. We fix $\varepsilon >0$ and
for every $\delta >0$ define%
\begin{equation*}
E_{\delta }:=\left\{ t\in \left( 0,1\right) :\psi \left( \left\vert
c^{2}A_{t,t+h}\right\vert ^{1/2}\right) >\left( 1-\varepsilon \right) h\text{
for some }h\in \left( 0,\delta \right) \right\} .
\end{equation*}%
where $c$ is the constant which appears in (\ref{CorLILAreaNew}). For fixed $%
t$ and $\delta $, $\mathbb{P}\left( t\in E_{\delta }\right) =1$ and by
Fubini $\mathbb{P}\left( \left\vert E_{\delta }\right\vert =1\right) =1$.
Now setting $E:=\cap _{\delta >0}E_{\delta }=\cap _{n\geq 1}E_{\frac{1}{n}}$
gives%
\begin{equation*}
\mathbb{P}\left( \left\vert E\right\vert =1\right) =1\text{.}
\end{equation*}%
This says that for every $t\in E$ there exists an arbitrary small $h>0$,
such that%
\begin{equation*}
\psi \left( \left\vert c^{2}A_{t,t+h}\right\vert ^{1/2}\right) >\left(
1-\varepsilon \right) h
\end{equation*}%
and these intervals of the form $\left[ t,t+h\right] $ form a Vitali
covering of $\left[ 0,1\right] .$ Hence we can pick a finite disjoint union
of such sets, each of length less than $\delta $ but of total length of at
least $1-\varepsilon .$ Now let $D$ be a dissection with mesh $\left\vert
D\right\vert <\delta $ which includes the above collection as subintervals.
Then%
\begin{eqnarray*}
V_{\psi \circ \sqrt{.}}^{D}\left( c^{2}A\right) &\geq &\sum_{i:t_{i}\in
D}\psi \left( \left\vert c^{2}A_{t_{i},t_{i+1}}\right\vert ^{1/2}\right) \\
&\geq &\sum_{i:t_{i}\text{ is an element of the subcovering}}\psi \left(
\left\vert c^{2}A_{t_{i},t_{i+1}}\right\vert ^{1/2}\right) \\
&\geq &\left( 1-\varepsilon \right) \sum_{i:t_{i}\text{ is an element of the
subcovering}}h_{i} \\
&\geq &\left( 1-\varepsilon \right) ^{2}
\end{eqnarray*}%
Now we see for each $\varepsilon >0$ and $\delta >0$%
\begin{equation*}
P\left( \sup_{\left\vert D\right\vert <\delta }V_{\psi \circ \sqrt{.}%
}^{D}\left( c^{2}A\right) >\left( 1-\varepsilon \right) ^{2}\right) =1
\end{equation*}%
Letting $\varepsilon \rightarrow 0$ and then $\delta \rightarrow 0$ (through
a countable sequence) gives%
\begin{equation*}
\mathbb{P}\left( \lim_{\delta \downarrow 0}\sup_{\left\vert D\right\vert
<\delta }V_{\psi \circ \sqrt{.}}^{D}\left( c^{2}A\right) \geq 1\right) =1.
\end{equation*}%
It is elementary to check from the definition of $\psi $ that for all $c\geq
0$ there exists a $\Delta _{c}$ so that%
\begin{equation*}
\forall s\geq 0:\psi \left( cs\right) \leq \Delta _{c}\psi \left( s\right)
\end{equation*}%
This readily implies that%
\begin{equation*}
\mathbb{P}\left( \lim_{\delta \downarrow 0}\sup_{\left\vert D\right\vert
<\delta }V_{\psi \circ \sqrt{.}}^{D}\left( A\right) \geq \frac{1}{\Delta _{c}%
}\right) =1
\end{equation*}%
and the proof is finished.
\end{proof}

As a corollary of this we can now give the proof of the regularity and
integrability of L\'{e}vy's area as stated in theorem \ref{ThmLevyArea} in
the introduction.

\begin{proof}
(Theorem \ref{ThmLevyArea}) The Gauss tail of $\left\vert A\right\vert
_{\psi \text{-var;}\left[ 0,1\right] }$ is an obvious consequence of our
general theorem \ref{ThGausstailsXPsvar} applied to $d$-dimensional Brownian
motion. Too see optimality take a $\tilde{\psi}$ such that%
\begin{equation}
\lim_{x\rightarrow 0}\frac{\tilde{\psi}\left( x\right) }{\psi _{p,2}\left(
x\right) }=\infty .  \label{EqstrongerPsiNorm}
\end{equation}%
Fix a dissection $D=\left( t_{i}\right) \subset \left[ 0,1\right] $ and
write $a_{i}=\left\vert A_{t_{i},t_{i+1}}\right\vert ^{1/2}$ for brevity.
Then 
\begin{equation*}
\sum_{i}\psi _{p,2}\left( a_{i}\right) =\sum_{i}\tilde{\psi}\left(
a_{i}\right) \frac{\psi _{p,2}\left( a_{i}\right) }{\tilde{\psi}\left(
a_{i}\right) }\leq \left( \sup_{i}\frac{\psi _{p,2}\left( a_{i}\right) }{%
\tilde{\psi}\left( a_{i}\right) }\right) \sum_{i}\tilde{\psi}\left(
a_{i}\right) .
\end{equation*}%
Thanks to the previous theorem, taking $\lim_{\delta \downarrow
0}\sup_{\left\vert D\right\vert \leq \delta }$ leaves the left-hand-side
above strictly positive. By a.s.\ (uniform) continuity of $\left( s,t\right)
\mapsto A_{s,t}$ we see that $a_{i}$ will be arbitrarily small as $%
\left\vert D\right\vert \downarrow 0$, uniformly over all $i$. Thus, $%
\sup_{i}\psi _{p,2}\left( a_{i}\right) /\tilde{\psi}\left( a_{i}\right)
\rightarrow 0$ with $\left\vert D\right\vert \downarrow 0$ and so we must
indeed have that 
\begin{equation*}
\lim_{\delta \downarrow 0}\sup_{\left\vert D\right\vert \leq \delta }\sum_{i}%
\tilde{\psi}\left( a_{i}\right) =+\infty \text{ a.s.}
\end{equation*}%
This finishes the proof of Theorem \ref{ThmLevyArea}.
\end{proof}

\subsection{Regularity of Iterated Integrals, Stochastic Integrals and
Solutions to stochastic differential equations in $\protect\psi $-variation 
\label{SubSectionroughpathinpsivariation}}

We now show that typical rough path estimates remain valid in $\psi $%
-variation. $G^{N}\left( \mathbb{R}^{d}\right) $ denotes the free step-$N$
nilpotent group with $d$-generators and is the natural state space for rough
paths. Equipped with\ Carnot-Caratheodory metric, it forms a metric space
and the concept of $\psi $-variation is immediately meaningful. One of the
basic theorems in rough path theory says that a continuous $G^{\left[ p%
\right] }\left( \mathbb{R}^{d}\right) $-valued path $\mathbf{x}\left( \cdot
\right) $ of finite $p$-variation lifts uniquely to a $G^{N}\left( \mathbb{R}%
^{d}\right) $-valued path, say $S_{N}\left( \mathbf{x}\right) $, also of
finite $p$-variation, for any $N\geq \left[ p\right] $ and for all $\left[
s,t\right] \subset \left[ 0,1\right] $%
\begin{equation}
\left\Vert S_{N}\left( \mathbf{x}\right) _{s,t}\right\Vert \leq \left\Vert
S_{N}\left( \mathbf{x}\right) \right\Vert _{p\text{-var;}\left[ s,t\right]
}\leq C\left( N,p\right) \left\Vert \mathbf{x}\right\Vert _{p\text{-var;}%
\left[ s,t\right] }.  \label{TerryLiftingEstimatePvar}
\end{equation}%
We now show how to extend this to $\psi $-variation. The proof is a based on
the clever concept of control functions introduced in the work of Lyons.
These are continous functions $\omega :\left\{ \left( s,t\right) :0\leq
s\leq t\leq 1\right\} \rightarrow \lbrack 0,\infty )$, zero on the diagonal,
and super-additive in the sense that $\omega \left( s,t\right) +\omega
\left( t,u\right) \leq \omega \left( s,u\right) $ for all $0\leq s\leq t\leq
u\leq 1$. The point is that\footnote{%
The concept makes sense for any continuous path with values in a metric
space.}%
\begin{equation*}
d\left( x_{s,}x_{t}\right) ^{p}\leq \omega \left( s,t\right)
\Longleftrightarrow \left\vert x\right\vert _{p\text{-var;}\left[ s,t\right]
}^{p}\leq \omega \left( s,t\right) .
\end{equation*}%
Examples of control functions to have in mind are $\left\vert t-s\right\vert
^{\theta }$ for $\theta \geq 1$ and $\left\vert x\right\vert _{p\text{-var;}%
\left[ s,t\right] }^{p}$ as well as $V_{\psi ;\left[ s,t\right] }\left(
x\right) $.

\begin{theorem}
Let $\mathbf{x}\in C^{\psi \text{-var}}\left( \left[ 0,1\right] ,G^{\left[ p%
\right] }\left( \mathbb{R}^{d}\right) \right) $ and $\psi :\left[ 0,\infty
\right) \rightarrow \left[ 0,\infty \right) ,$ $\psi \left( 0\right) =0$,
continuous, strictly increasing and onto. Assume that $\psi ^{-1}\left(
\cdot \right) ^{p}$ is convex. Then, for $N\geq \left[ p\right] $ with $%
C=C\left( N,p\right) $,%
\begin{equation*}
\left\Vert S_{N}\left( \mathbf{x}\right) \right\Vert _{\psi \text{-var;}%
\left[ s,t\right] }\leq C\left\Vert \mathbf{x}\right\Vert _{\psi \text{-var;}%
\left[ s,t\right] }.
\end{equation*}%
This estimate remains true with $C=C\left( N,p,\varepsilon \right) $ if $%
\psi ^{-1}\left( \cdot \right) ^{p}$ is only assumed convex on $\left[
0,\varepsilon \right] $ and $\psi $ is such that $\psi \left( \eta s\right)
\leq \Delta _{\eta }\psi \left( s\right) $ for $s$ near $0+$ and $\Delta
_{\eta }\downarrow 0$ as $\eta \downarrow 0$. (These conditions are
satisfied for any $\psi \equiv \psi _{p,2}$ near $0+$.)
\end{theorem}

\begin{proof}
\underline{Step 1:} We first assume that $\psi ^{-1}\left( \cdot \right)
^{p} $ is convex on $[0,\infty )$. Fix $\left[ s,t\right] \subset \left[ 0,1%
\right] $ and define for $u,v\in \left[ s,t\right] $,%
\begin{equation*}
\omega \left( u,v\right) :=\sup_{D\subset \left[ u,v\right]
}\sum_{i:t_{i}\in D}\psi \left( \frac{\left\Vert
x_{t_{i},t_{i+1}}\right\Vert }{\left\Vert x\right\Vert _{\psi \text{-var,}%
\left[ s,t\right] }}\right) \text{.}
\end{equation*}%
Clearly, $\omega $ is a control function on $\left[ s,t\right] $ and from
the definition of $\psi $-variation it follows that $\omega \left(
s,t\right) \leq 1$. It is also clear from the definition that%
\begin{equation*}
\psi \left( \frac{\left\Vert x_{u,v}\right\Vert }{\left\Vert x\right\Vert
_{\psi \text{-var,}\left[ s,t\right] }}\right) \leq \omega \left( u,v\right)
\end{equation*}%
and this can be rewritten as $\left\Vert x_{u,v}\right\Vert ^{p}\leq
\left\Vert x\right\Vert _{\psi \text{-var,}\left[ s,t\right] }^{p}\left(
\psi ^{-1}\left( \omega \left( u,v\right) \right) \right) ^{p}$ and since
the upper bound is a control function in $\left( u,v\right) $ this yields 
\begin{equation*}
\left\Vert x\right\Vert _{p\text{-var;}\left[ u,v\right] }^{p}\leq
\left\Vert x\right\Vert _{\psi \text{-var,}\left[ s,t\right] }^{p}\left(
\psi ^{-1}\left( \omega \left( u,v\right) \right) \right) ^{p}.
\end{equation*}%
Using (\ref{TerryLiftingEstimatePvar}) (thanks to the convexity assumption,
the right-hand-side is control function!) we get%
\begin{equation*}
\left\Vert S_{N}\left( x\right) _{u,v}\right\Vert \leq c_{1}\left\Vert
x\right\Vert _{\psi \text{-var,}\left[ s,t\right] }\left( \psi ^{-1}\left(
\omega \left( u,v\right) \right) \right) \text{,}
\end{equation*}%
with $c_{1}=C\left( N,p\right) $. Using this estimate and the superaddivity
of $\omega $ to see that%
\begin{eqnarray*}
\sup_{D\subset \left[ s,t\right] }\sum_{i:t_{i}\in D}\psi \left( \frac{%
\left\Vert S_{N}\left( x\right) _{t_{i},t_{i+1}}\right\Vert }{%
c_{1}\left\Vert x\right\Vert _{\psi \text{-var,}\left[ s,t\right] }}\right)
&\leq &\sup_{D\cap \left[ s,t\right] }\sum_{i:t_{i}\in D}\omega \left(
t_{i},t_{i+1}\right) \\
&\leq &\omega \left( s,t\right) \leq 1\text{.}
\end{eqnarray*}%
By the definition of $\psi $-variation we conclude that $\left\Vert
S_{N}\left( x\right) \right\Vert _{\psi \text{-var,}\left[ s,t\right] }\leq
c_{1}\left\Vert x\right\Vert _{\psi \text{-var,}\left[ s,t\right] }$.\newline
\underline{Step-2:} We now assume that $\psi ^{-1}\left( \cdot \right) ^{p}$
is convex only on some interval $\left[ 0,\varepsilon \right] \subset \left[
0,1\right] $. Again, fix $\left[ s,t\right] \subset \left[ 0,1\right] $ and
define for $u,v\in \left[ s,t\right] $,%
\begin{equation*}
\omega \left( u,v\right) :=\sup_{D\subset \left[ u,v\right]
}\sum_{i:t_{i}\in D}\psi \left( \frac{\left\Vert
x_{t_{i},t_{i+1}}\right\Vert }{M\left\Vert x\right\Vert _{\psi \text{-var,}%
\left[ s,t\right] }}\right)
\end{equation*}%
where $M$ is chosen large enough so that $\omega \left( s,t\right) \leq
\varepsilon $. (This is possible by our assumption on $\Delta _{\eta }.$)
Again, $\omega $ is a control function on $\left[ s,t\right] $ and arguing
exactly as in step 1, using that $\psi ^{-1}\left( \omega \left( u,v\right)
\right) ^{p}$ is a also control on $\left[ s,t\right] $, a we led%
\begin{eqnarray*}
\sup_{D\subset \left[ s,t\right] }\sum_{i:t_{i}\in D}\psi \left( \frac{%
\left\Vert S_{N}\left( x\right) _{t_{i},t_{i+1}}\right\Vert }{%
c_{1}M\left\Vert x\right\Vert _{\psi \text{-var,}\left[ s,t\right] }}\right)
&\leq &\sup_{D\cap \left[ s,t\right] }\sum_{i:t_{i}\in D}\omega \left(
t_{i},t_{i+1}\right) \\
&\leq &\omega \left( s,t\right) \leq \varepsilon \leq 1\text{.}
\end{eqnarray*}%
By the definition of $\psi $-variation we conclude that $\left\Vert
S_{N}\left( x\right) \right\Vert _{\psi \text{-var,}\left[ s,t\right] }\leq
c_{1}M\left\Vert x\right\Vert _{\psi \text{-var,}\left[ s,t\right] }$.%
\newline
\end{proof}

\bigskip

\begin{corollary}
Let $X$ be a Gaussian process satisfying the assumptions of Theorem \ref%
{ThGausstailsXPsvar} and let $\mathbf{X}$ denote the corresponding rough
path. Then $S_{N}\left( \mathbf{X}\right) $ has finite $\psi _{p,2}$%
-variation for every $N>2$ and $\left\Vert S_{N}\left( \mathbf{X}\right)
\right\Vert _{\psi _{p,2}\text{-var;}\left[ 0,1\right] }$ has a Gauss tail.
Applied to $d$-dimensional Brownian motion, this says that Brownian motion
and all its iterated Stratonovich integrals up to order $N$, written as $%
S_{N}\left( \mathbf{B}\right) $ and viewed as a diffusion in the step-$N$
nilpotent group with $d$ generators (e.g.\ \cite{benarous-89, lyons-98}),
have Gaussian integrability in the sense that $\left\Vert S_{N}\left( 
\mathbf{B}\right) \right\Vert _{\psi _{2,2}\text{-var;}\left[ 0,1\right]
}<\infty $ has a Gauss tail.
\end{corollary}

\begin{remark}
To write this result in its most explicit form, take a multi-index $I=\left(
i_{1},\dots ,i_{N}\right) $ $\in \left\{ 1,\dots ,d\right\} ^{N}$ and 
\begin{eqnarray*}
\Delta _{\left[ t_{i},t_{i+1}\right] }^{N} &\equiv &\left\{ \left(
s_{1},\dots ,s_{N}\right) :t_{i}\leq s_{1}<\dots <t_{N}<t_{i+1}\right\} 
\text{,} \\
\circ \mathrm{d}B^{I} &\equiv &\circ \mathrm{d}B^{i_{1}}\circ \dots \circ 
\mathrm{d}B^{i_{N}}.
\end{eqnarray*}%
The Gauss tail of $\left\Vert S_{N}\left( \mathbf{B}\right) \right\Vert
_{\psi \text{-var;}\left[ 0,1\right] }<\infty $ is equivalent to the Gauss
tail of 
\begin{equation*}
\inf \left\{ \varepsilon >0:\sup_{D\subset \left[ 0,1\right]
}\sum_{i:t_{i}\in D}\psi \left( \left\vert \frac{\int_{\Delta _{\left[
t_{i},t_{i+1}\right] }^{N}}\circ \mathrm{d}B^{I}}{\varepsilon ^{\left\vert
I\right\vert }}\right\vert ^{\frac{1}{\left\vert I\right\vert }}\right) \leq
1\right\}
\end{equation*}%
for all possible multi-indices $I\in \left\{ 1,\dots ,d\right\} ^{N}.$
\end{remark}

\begin{remark}
(1) The same argument shows that solutions to RDEs (rough differential
equations) driven by some $\mathbf{X}$ with finite $\psi _{p,2}$-variation
also have finite $\psi _{p,2}$-variation (and the precise estimates from 
\cite{friz-victoir-06-Euler} also yields integrability properties for the
corresponding $\psi _{p,2}$-seminorm). \newline
(2) If one only considers the case of driving Brownian motion and the RDE
solution as a path in Euclidean space (rather than a rough path in its own
right), finite $\psi _{2,2}$-variation is easy to see directly: Under
standard regularity assumptions, RDE solutions equal Stratonovich solutions
and are seen to be continuous semimartingales. But every continuous local
martingale is a time-change of Brownian motion (therefore of finite $\psi
_{2,2}$-variation); the bounded variation part is of course harmless.\newline
(3) The previous argument is not robust and relies crucially on the
semimartingale nature of the RDE solution. There is a variety of examples of
differential equations driven by Gaussian (and also Markovian) signals with
non-semimartingale solutions. Our results imply in particular $\psi _{p,2}$%
-variation regularity of sample paths is a robust property and valid beyond
the usual semimartingale context.
\end{remark}

\appendix

\section{Appendix: Variation norms\label{AppendixVariationNorms}}

The original definitions of \cite{MusielakOrlicz59:OngeneralizedVariationI}
and \cite{LesniewiczOrlicz73:OngeneralizedVariationsII} extend immediately
to paths with values in a metric space $\left( E,d\right) $. The recent
monograph \cite{dudley-norvaisa-99} contains extensive references to $\psi $%
-variation.

\begin{definition}
Let $x\in C\left( \left[ 0,1\right] ,E\right) $ and $\psi :\left[ 0,\infty
\right) \rightarrow \left[ 0,\infty \right) ,$ $\psi \left( 0\right) =0$,
continuous, strictly increasing and onto\footnote{%
These assumptions can be relaxed but allow us to consider $\psi ^{-1}$
without added technicalities.}. Then we define the $\psi $-variation "norm"
of $x$ over $\left[ s,t\right] $ by 
\begin{equation*}
\left\vert x\right\vert _{\psi \text{-var,}\left[ s,t\right] }:=\inf
\{\varepsilon >0:\underset{=:V_{\psi \text{;}\left[ s,t\right]
}^{\varepsilon }\left( x\right) }{\underbrace{\sup_{D\subset \left[ s,t%
\right] }\sum_{i:t_{i}\in D}\psi \left( \frac{d\left(
x_{t_{i}},x_{t_{i+1}}\right) }{\varepsilon }\right) }}\leq 1\}\text{.}
\end{equation*}%
We write $V_{\psi \text{,}\left[ s,t\right] }\left( x\right) $ rather than $%
V_{\psi \text{,}\left[ s,t\right] }^{\varepsilon }\left( x\right) $ when $%
\varepsilon =1$ and also define%
\begin{equation*}
C^{\psi \text{-var}}\left( \left[ 0,1\right] ,E\right) :=\left\{ x\in
C\left( \left[ 0,1\right] ,E\right) :\left\vert x\right\vert _{\psi \text{%
-var;}\left[ 0,1\right] }<\infty \right\} .
\end{equation*}
\end{definition}

It may help to keep in mind that for $\psi \left( s\right) =s^{p}$ this is
consistent with%
\begin{equation*}
V_{\left( \cdot \right) ^{p};\left[ s,t\right] }\left( x\right) \equiv
\left\vert x\right\vert _{p\text{-var;}\left[ s,t\right] }^{p}:=\sup_{D%
\subset \left[ s,t\right] }\sum_{i:t_{i}\in D}d\left(
x_{t_{i}},x_{t_{i+1}}\right) ^{p}.
\end{equation*}

\begin{definition}
We say that a function $\psi :\left[ 0,\infty \right) \rightarrow \left[
0,\infty \right) $ has the property ($d_{2}^{loc}$), if 
\begin{equation*}
\lim \sup_{s\downarrow 0}\psi ^{-1}\left( 2\psi \left( s\right) \right)
/s<\infty .
\end{equation*}%
Equivalently, for any $T>0$ there exists a constant $d_{2}=d_{2}\left(
T\right) $ such that $\forall s\in \left[ 0,T\right] :2\psi \left( s\right)
\leq \psi \left( d_{2}s\right) $.\newline
We say that a function $\psi :\left[ 0,\infty \right) \rightarrow \left[
0,\infty \right) $ has the property ($\Delta _{2}^{loc}$), if 
\begin{equation*}
\lim \sup_{s\downarrow 0}\psi \left( 2\psi \left( s\right) \right) /s<\infty
.
\end{equation*}%
Equivalently, for any $T>0$ there exists a constant $\Delta _{2}=\Delta
_{2}\left( T\right) $ such that $\forall s\in \left[ 0,T\right] :\psi \left(
2s\right) \leq \Delta _{2}\psi \left( s\right) $. \newline
We speak of properties $d_{2}$ and $\Delta _{2}$ if the $\lim
\sup_{s\downarrow 0}$ can be replaced by $\sup_{s\in \mathbb{R}^{+}}$.
\end{definition}

Assuming the above properties, the notion of $\psi $-variation behaves "as
expected".

\begin{proposition}
Let $\psi :\left[ 0,\infty \right) \rightarrow \left[ 0,\infty \right) ,$ $%
\psi \left( 0\right) =0$, continuous, strictly increasing, onto and let $%
x\in C\left( \left[ 0,1\right] ,E\right) $. Assume further that $\psi $
fulfills the properties ($d_{2}^{loc})$ and ($\Delta _{2}^{loc}$). Then%
\begin{equation*}
V_{\psi ;\left[ 0,1\right] }\left( x\right) <\infty \Leftrightarrow
\left\vert x\right\vert _{\psi \text{-var;}\left[ 0,1\right] }<\infty .
\end{equation*}
\end{proposition}

\begin{proof}
Using ($d_{2}^{loc}$) with the result constant $d_{2}=d_{2}\left( T\right) $
with $T=1$ say (for $\varepsilon $ large enough), 
\begin{eqnarray*}
2\lim_{\varepsilon \rightarrow \infty }V_{\psi ;\left[ 0,1\right]
}^{\varepsilon }\left( x\right) &\leq &\lim_{\varepsilon \rightarrow \infty
}V_{\psi ;\left[ 0,1\right] }^{\varepsilon /d_{2}}\left( x\right) \\
&=&\lim_{\varepsilon \rightarrow \infty }V_{\psi ;\left[ 0,1\right]
}^{\varepsilon }\left( x\right) \\
&\leq &V_{\psi ;\left[ 0,1\right] }\left( x\right) <\infty
\end{eqnarray*}%
It follows that $\lim_{\varepsilon \rightarrow \infty }V_{\psi ;\left[ 0,1%
\right] }^{\varepsilon }\left( x\right) =0$ and in particular there exists $%
\varepsilon ^{\ast }$ such that $V_{\psi ;\left[ 0,1\right] }^{\varepsilon
^{\ast }}\left( x\right) \leq 1$. To see the other direction, we make use of
($\Delta _{2}^{loc}$). Of course, we may assume $\left\vert x\right\vert
_{\psi \text{-var;}\left[ 0,1\right] }>0$ and note that ($\Delta _{2}^{loc}$%
) implies for every $c\geq 0$ and every $T\geq 0$ the existence of a
constant $\Delta _{c}$ such that $\psi \left( cs\right) \leq \Delta _{c}\psi
\left( s\right) $ for all $s\in \left[ 0,T\right] .$Choosing $T=\frac{%
\left\vert x\right\vert _{0,\left[ 0,1\right] }}{\left\vert x\right\vert
_{\psi \text{-var;}\left[ 0,1\right] }}$ we can estimate 
\begin{eqnarray*}
\sum_{i}\psi \left( d\left( x_{t_{i},t_{i+1}}\right) \right) &\leq
&\sum_{i}\psi \left( \frac{d\left( x_{t_{i},t_{i+1}}\right) }{\left\vert
x\right\vert _{\psi \text{-var;}\left[ 0,1\right] }}\left\vert x\right\vert
_{\psi \text{-var;}\left[ 0,1\right] }\right) \\
&\leq &\Delta _{c}\sum_{i}\psi \left( \frac{d\left( x_{t_{i},t_{i+1}}\right) 
}{\left\vert x\right\vert _{\psi \text{-var;}\left[ 0,1\right] }}\right)
\leq \Delta _{c}
\end{eqnarray*}%
with $c=\left\vert x\right\vert _{\psi \text{-var;}\left[ 0,1\right] }$.
Taking the supremum over all dissections $D=\left( t_{i}\right) \subset %
\left[ 0,1\right] $ the result follows.
\end{proof}

\begin{lemma}
\label{LemmaIncrementSmallerNorm}Let $\psi :\left[ 0,\infty \right)
\rightarrow \left[ 0,\infty \right) ,$ $\psi \left( 0\right) =0$,
continuous, strictly increasing, onto and let $x\in C\left( \left[ 0,1\right]
,E\right) $. Then, for all $0\leq s<t\leq 1$,%
\begin{equation*}
\left\vert x_{s,t}\right\vert \leq \psi ^{-1}\left( 1\right) \left\vert
x\right\vert _{\psi \text{-var,}\left[ s,t\right] }\text{.}
\end{equation*}
\end{lemma}

\begin{proof}
Clearly,%
\begin{equation*}
\frac{\left\vert x_{s,t}\right\vert }{\left\vert x\right\vert _{\psi \text{%
-var,}\left[ s,t\right] }}\leq \psi ^{-1}\left( V_{\psi \text{-var,}\left[
s,t\right] }\left( \frac{x}{\left\vert x\right\vert _{\psi \text{-var,}\left[
s,t\right] }}\right) \right)
\end{equation*}%
and by definition of $V_{\psi }$ it follows that $\left\vert
x_{s,t}\right\vert \leq \psi ^{-1}\left( 1\right) \left\vert x\right\vert
_{\psi \text{-var,}\left[ s,t\right] }$ as required.
\end{proof}

\section{Appendix: Rough Paths\label{AppendixRoughPaths}}

\subsection{Notation\label{SubsectionNotation}}

Complete expositions of rough path theory include \cite{lyons-98}, \cite%
{lyons-qian-02}. For the reader's convenience we review some basic concepts.
For a fixed $N\geq 0$, the truncated tensor algebra $T^{N}\left( \mathbb{R}%
^{d}\right) $ of degree $N$ is defined as 
\begin{equation*}
T^{N}\left( \mathbb{R}^{d}\right) :=\oplus _{k=0}^{N}\left( \mathbb{R}%
^{d}\right) ^{\otimes k}\text{,}
\end{equation*}%
where we set $\left( \mathbb{R}^{d}\right) ^{0}:=\mathbb{R}^{d}$ and $\left( 
\mathbb{R}^{d}\right) ^{\otimes k}$, equipped with the usual tensor product.
Then $T^{N}\left( \mathbb{R}^{d}\right) $ forms an algebra under the usual
scalar product and vector addition. Using Young integrals we can define for
a path $x\in C^{p\text{-var}}\left( \left[ 0,1\right] ,\mathbb{R}^{d}\right) 
$ with $p\in \left[ 1,2\right) $ its step-$N$ lift by%
\begin{eqnarray*}
S_{N}\left( x\right) :[0,1] &\rightarrow &T^{N}\left( \mathbb{R}^{d}\right)
\\
t &\mapsto &1+\sum_{i=i}^{N}\int_{0<s_{1}<\cdots <s_{k}<t}\mathrm{d}x\otimes
\cdots \otimes \mathrm{d}x
\end{eqnarray*}%
This yields a path $S_{N}\left( x\right) $ with values in $T^{N}\left( 
\mathbb{R}^{d}\right) $. There are numerous non-linear relations between
these iterated integrals and $S_{N}\left( x\right) $ stays in a the proper
subset 
\begin{equation*}
G^{N}\left( \mathbb{R}^{d}\right) :=\left\{ \mathbf{g}\in T^{N}\left( 
\mathbb{R}^{d}\right) :\exists x\text{ smooth:\ }S_{N}\left( x\right) _{1}=%
\mathbf{g}\right\} \text{.}
\end{equation*}%
It turns out that $G^{N}\left( \mathbb{R}^{d}\right) $ is a Lie group under
induced tensor multiplication and (a possible realization of) the step-$N$
free nilpotent group with $d$ generators. The Carnot-Caratheodory (CC) "norm"%
\begin{equation*}
\left\Vert \mathbf{g}\right\Vert =\inf \left\{ \int_{0}^{1}\left\vert \dot{x}%
_{u}\right\vert \mathrm{d}u:x\text{ smooth,\ }S_{N}\left( x\right) _{1}=%
\mathbf{g}\right\}
\end{equation*}%
induces a (left-invariant) metric, the CC\ metric $d:\left( \mathbf{g},%
\mathbf{h}\right) \mapsto \left\Vert \mathbf{g}^{-1}\otimes \mathbf{h}%
\right\Vert $. The CC norm is homogenous under dilation%
\begin{equation*}
\mathbf{g}=1+\mathbf{g}^{1}+\cdots +\mathbf{g}^{N}\mapsto 1+\lambda \mathbf{g%
}^{1}+\dots \lambda ^{N}\mathbf{g}^{N}\text{.}
\end{equation*}%
and using equivalence of continuous, homogenous norms (by a compactness
argument) one the Lipschitz equivalence%
\begin{equation*}
\left\Vert \mathbf{g}\right\Vert \sim \left\vert \mathbf{g}^{1}\right\vert
+\left\vert \mathbf{g}^{2}\right\vert ^{1/2}+\dots +\left\vert \mathbf{g}%
^{N}\right\vert ^{1/N}
\end{equation*}%
which is very useful in compuations. Having a metric structure on $%
G^{N}\left( \mathbb{R}^{d}\right) $ we can speak paths of finite $p$- or $%
\psi $-variation, as well as of $\alpha $-H\"{o}lder or $\varphi $-H\"{o}%
lder regularity. The class of (weak geometric) $p$-rough path is defined as%
\begin{equation*}
C^{p\text{-var}}\left( \left[ 0,1\right] ,G^{\left[ p\right] }\left( \mathbb{%
R}^{d}\right) \right) .
\end{equation*}

The order of nilpotency is tied to the regularity of the path. In the
example of Brownian motion and L\'{e}vy's area (which give rise to a path in
the step-$2$ group, one can take $p\in \left( 2,3\right) $. Rough path
theory implies that $p$-rough paths drive differential equations in a fully
determinstic way with a variety of continuity properties of the solution map.

\subsection{Translating Rough paths}

We are interested in the translates of rough paths by paths of a
Cameron-Martin space. For Brownian motion it is obvious that Cameron-Martin
paths are of finite $1$-variation. On the other hand, we have seen that
(theorem \ref{Th2dimVarCMspace}) for many Gaussian processes $X$ one has $%
\mathcal{H}\left( X\right) \hookrightarrow C^{\rho \text{-var}}$ for some $%
\rho $ where $\mathcal{H}\left( X\right) $ is the Cameron-Martin space
associated to $X$. We recall the definition of the translation operator in
rough path theory (cf.\ \cite{lyons-qian-02}).

\begin{definition}
For $\mathbf{x}\in C^{p\text{-var}}\left( \left[ 0,1\right] ,G^{2}\left( 
\mathbb{R}^{d}\right) \right) ,\,p\geq 1$ and $h\in C^{\rho \text{-var}%
}\left( \left[ 0,1\right] ,\mathbb{R}^{d}\right) $ such that $1/p+1/\rho >1$%
, we define the translation operator $T_{h}$ by%
\begin{equation*}
T_{h}\left( \mathbf{x}\right) _{s,t}=1+\left( \mathbf{x}_{s,t}^{1}+h_{s,t}%
\right) +\left( \mathbf{x}_{s,t}^{2}+\int_{s}^{t}x_{s,u}\otimes \mathrm{d}%
h_{u}+\int_{s}^{t}h_{s,u}\otimes \mathrm{d}x_{u}+\int_{s}^{t}h_{s,u}\otimes 
\mathrm{d}h_{u}\right)
\end{equation*}%
where all cross integrals are well-defined Young integrals.
\end{definition}

One should note that for smooth paths~$x,h$ and $\mathbf{x}=S_{2}\left(
x\right) $, the step-$2$ lift of $x$, the translate $T_{h}\left( \mathbf{x}%
\right) $ is precisely $S_{2}\left( x+h\right) $.

\begin{theorem}
\label{ThTranslationPhiVarAndRhoVar}Assume $p\geq 1$ and $1/p+1/\rho >1$.
Let $\psi :\left[ 0,\infty \right) \rightarrow \left[ 0,\infty \right) ,$ $%
\psi \left( 0\right) =0$, continuous, strictly increasing and onto. Assume
that $\psi $ satisfies the property ($d_{2}^{loc}$) and that $\psi $ is such
that $\psi \left( \eta s\right) \leq \Delta _{\eta }\psi \left( s\right) $
for every $\eta >0$, $s$ near $0+$ and that $\Delta _{\eta }\downarrow 0$ as 
$\eta \downarrow 0$. Furthermore, assume

\begin{enumerate}
\item \label{PropertyPsicomposedRootconvex}The function $x\mapsto \psi
\left( x^{1/\rho }\right) $ is convex near $0+$

\item \label{PropertyPsiinversetothepowerConvex}The function $x\mapsto
\left( \psi ^{-1}\left( x\right) \right) ^{p}$ is convex near $0+$
\end{enumerate}

(These conditions are satisfied for any $\psi \equiv \psi _{p,2}$ near $0+$%
.) Then there exists a constant $C>0$ such that for all $\mathbf{y}\in
C^{\psi \text{-var}}\left( \left[ 0,1\right] ,G^{2}\left( \mathbb{R}%
^{d}\right) \right) $ and $h\in C^{\rho \text{-var}}\left( \left[ 0,1\right]
,\mathbb{R}^{d}\right) $%
\begin{equation}
\left\Vert T_{h}\left( \mathbf{y}\right) \right\Vert _{\psi -\text{var;}%
\left[ s,t\right] }\leq C\left( \left\Vert \mathbf{y}\right\Vert _{\psi -%
\text{var;}\left[ s,t\right] }+\left\vert h\right\vert _{\rho \text{-var;}%
\left[ s,t\right] }\right) .  \label{EqNormTranslateRoughPathbyCM}
\end{equation}
\end{theorem}

\begin{proof}
Note that (\ref{EqNormTranslateRoughPathbyCM}) is equivalent (by homogenity
of the norm) to the statement that for all $\lambda >0$ small enough
(depending on $\mathbf{y},h$ and $\psi $)%
\begin{equation*}
\left\Vert T_{\lambda h}\left( \delta _{\lambda }\mathbf{y}\right)
\right\Vert _{\psi -\text{var;}\left[ s,t\right] }\leq C\left( \left\Vert
\delta _{\lambda }\mathbf{y}\right\Vert _{\psi -\text{var;}\left[ s,t\right]
}+\left\vert \lambda h\right\vert _{\psi \text{-var;}\left[ s,t\right]
}\right) \text{.}
\end{equation*}%
Therefore we can assume in the rest of the proof that the variation "norms"
of $y$ and $h$ are arbitrary small (and therefore the increments as well,
cf.\ Lemma \ref{LemmaIncrementSmallerNorm}). By equivalence of homogenous
norms, the CC\ norm of the group element $T_{h}\left( \mathbf{y}\right)
_{s,t}$ is estimated by%
\begin{eqnarray*}
\left\Vert T_{h}\left( \mathbf{y}\right) _{s,t}\right\Vert &\lesssim
&\left\vert h_{s,t}+\mathbf{y}_{s,t}^{1}\right\vert \\
&&+\sqrt{\left\vert \mathbf{y}_{s,t}^{2}\right\vert }+\sqrt{\left\vert
\int_{s}^{t}h_{s,r}\otimes \mathrm{d}y_{r}\right\vert }+\sqrt{\left\vert
\int_{s}^{t}y_{s,r}\otimes \mathrm{d}h_{r}\right\vert }+\sqrt{\left\vert
\int_{s}^{t}h_{s,r}\otimes \mathrm{d}h_{r}\right\vert }.
\end{eqnarray*}%
By assumption $\psi $ is "quasi-subadditive" in the sense that 
\begin{equation*}
\psi \left( s+t\right) \leq \psi \left( 2\max \left( s,t\right) \right) \leq
\Delta _{2}\psi \left( \max \left( s,t\right) \right) \leq \Delta _{2}\left(
\psi \left( s\right) +\psi \left( t\right) \right) ,
\end{equation*}%
as long as $s,t\in \left[ 0,1\right] $, say. Absorbing such constants into $%
\lesssim $ we have 
\begin{eqnarray*}
\psi \left( \left\Vert T_{h}\mathbf{y}_{s,t}\right\Vert \right) &\lesssim
&\left( \psi \left( \left\Vert \mathbf{y}_{s,t}\right\Vert \right) +\psi
\left( \left\Vert S_{2}\left( h\right) _{s,t}\right\Vert \right) \right. \\
&&\left. +\psi \left( \sqrt{\left\vert \int_{s}^{t}h_{s,r}\otimes \mathrm{d}%
y_{r}\right\vert }\right) +\psi \left( \sqrt{\left\vert
\int_{s}^{t}y_{s,r}\otimes \mathrm{d}h_{r}\right\vert }\right) \right) .
\end{eqnarray*}%
By Young's inequality and using $\psi \left( cs\right) \leq \Delta _{c}\psi
\left( s\right) $ we see $\psi \left( \left\Vert S_{2}\left( h\right)
_{s,t}\right\Vert \right) \lesssim \psi \left( \left\vert h\right\vert
_{\rho -\text{var;}\left[ s,t\right] }\right) $. We now show that 
\begin{equation*}
\psi \left( \sqrt{\left\vert \int_{s}^{t}y_{s,r}\otimes \mathrm{d}%
h_{r}\right\vert }\right) \leq V_{\psi ;\left[ s,t\right] }\left( y\right)
+\psi \left( \left\vert h\right\vert _{\rho \text{-var};\left[ s,t\right]
}\right) .
\end{equation*}%
To this end, we note that $\left\vert y_{s,t}\right\vert \leq \psi
^{-1}\left( V_{\psi ;\left[ s,t\right] }\left( y\right) \right) $ and since
and since $\left( \psi ^{-1}\left( x\right) \right) ^{p}$ is convex near $0+$
the right hand side of $\left\vert y_{s,t}\right\vert ^{p}\leq \left[ \psi
^{-1}\left( V_{\psi ;\left[ s,t\right] }\left( y\right) \right) \right] ^{p}$
is a control from which it follows that $\left\vert y\right\vert _{p\text{%
-var;}\left[ s,t\right] }\leq \left[ \psi ^{-1}\left( V_{\psi ;\left[ s,t%
\right] }\left( y\right) \right) \right] $ and we rewrite this as%
\begin{equation*}
\psi (\left\vert y\right\vert _{p\text{-var;}\left[ s,t\right] })\leq \left(
V_{\psi ;\left[ s,t\right] }\left( y\right) \right) .
\end{equation*}%
Using Young's inequality ($1/p+1/\rho >1$), quasi-subadditivity, and the
estimate of the previous line, 
\begin{eqnarray}
\psi \left( \sqrt{\left\vert \int_{s}^{t}y_{s,r}\otimes \mathrm{d}%
h_{r}\right\vert }\right) &\leq &\psi \left( c\sqrt{\left\vert y\right\vert
_{p\text{-var;}\left[ s,t\right] }\left\vert h\right\vert _{\rho \text{-var};%
\left[ s,t\right] }}\right)  \notag \\
&\lesssim &\psi \left( \left\vert y\right\vert _{p\text{-var;}\left[ s,t%
\right] }\right) +\psi \left( \left\vert h\right\vert _{\rho \text{-var};%
\left[ s,t\right] }\right) \text{ }  \notag \\
&\leq &V_{\psi ;\left[ s,t\right] }\left( y\right) +\text{\ }\psi \left(
\left\vert h\right\vert _{\rho \text{-var};\left[ s,t\right] }\right) .
\label{estimate2ndlevel}
\end{eqnarray}%
What makes this estimate useful is that $V_{\psi ;\left[ s,t\right] }$ and $%
\psi (\left\vert h\right\vert _{\rho \text{-var};\left[ s,t\right] })$,
thanks to assumption \ref{PropertyPsicomposedRootconvex}, are controls and
hence super-additive in $\left( s,t\right) $. Replacing $s,t$ by some
arbitrary interval in a dissection $D$ of $\left[ s,t\right] $, following by
summation and taking $\sup_{D\subset \left[ s,t\right] }$ then gives%
\begin{eqnarray*}
V_{\psi ;\left[ s,t\right] }\left( \sqrt{\left\vert \int y_{.,r}\otimes 
\mathrm{d}h_{r}\right\vert }\right) &\equiv &\sup_{D\subset \left[ s,t\right]
}\sum_{i:t_{i}\in D}\psi \left( \sqrt{\left\vert
\int_{t_{i}}^{t_{i+1}}y_{t_{i},r}\otimes \mathrm{d}h_{r}\right\vert }\right)
\\
&\lesssim &V_{\psi ;\left[ s,t\right] }\left( \mathbf{y}\right) +\psi \left(
\left\vert h\right\vert _{\rho \text{-var};\left[ s,t\right] }\right) .
\end{eqnarray*}%
Of course, the same arguments applies to the other mixed integral,%
\begin{equation*}
V_{\psi ;\left[ s,t\right] }\left( \sqrt{\left\vert \int h_{.,r}\otimes 
\mathrm{d}y_{r}\right\vert }\right) \lesssim V_{\psi ;\left[ s,t\right]
}\left( \mathbf{y}\right) +\psi \left( \left\vert h\right\vert _{\rho \text{%
-var};\left[ s,t\right] }\right) .
\end{equation*}%
Thus, we have shown so far that%
\begin{equation}
V_{\psi ;\left[ s,t\right] }\left( T_{h}\left( \mathbf{y}\right) \right)
\lesssim V_{\psi ;\left[ s,t\right] }\left( \mathbf{y}\right) +\psi \left(
\left\vert h\right\vert _{\rho \text{-var};\left[ s,t\right] }\right) .
\label{EqVpsiTranslateEstimate}
\end{equation}%
To prepare for the remainder of the argument consider $a_{\varepsilon
},b_{\varepsilon },c_{\varepsilon }$ all decreasing in $\varepsilon \in
\left( 0,\infty \right) $ such that $a_{\varepsilon }\leq b_{\varepsilon
}+c_{\varepsilon }\forall \varepsilon $. Then%
\begin{eqnarray}
\inf \left\{ \varepsilon >0:a_{\varepsilon }\leq 1\right\} &\leq &\inf
\left\{ \varepsilon >0:b_{\varepsilon }+c_{\varepsilon }\leq 1\right\} 
\notag \\
&\leq &\inf \left\{ \varepsilon >0:b_{\varepsilon }\leq 1/2\right\}  \notag
\\
&&\vee \inf \left\{ \varepsilon >0:c_{\varepsilon }\leq 1/2\right\} .
\label{EqInfimum}
\end{eqnarray}%
We now use (\ref{EqVpsiTranslateEstimate}), (\ref{EqInfimum}),%
\begin{eqnarray*}
\allowbreak \left\Vert T_{h}\left( \mathbf{y}\right) \right\Vert _{\psi -%
\text{var;}\left[ s,t\right] } &\equiv &\inf \left( \varepsilon >0:V_{\psi 
\text{;}\left[ s,t\right] }\left( \delta _{1/\varepsilon }\left( T_{h}\left( 
\mathbf{y}\right) \right) \right) \leq 1\right) \\
&\leq &\inf \left( \varepsilon >0:c_{1}V_{\psi \text{;}\left[ s,t\right]
}\left( \delta _{1/\varepsilon }\mathbf{y}\right) +c_{1}\psi \left(
\left\vert h\right\vert _{\rho \text{-var};\left[ s,t\right] }/\varepsilon
\right) \leq 1\right) \\
&\leq &\inf \left( \varepsilon >0:c_{1}V_{\psi \text{;}\left[ s,t\right]
}\left( \delta _{1/\varepsilon }\mathbf{y}\right) \leq \frac{1}{2}\right) \\
&&\vee \inf \left( \varepsilon >0:c_{1}\psi \left( \left\vert h\right\vert
_{\rho \text{-var};\left[ s,t\right] }/\varepsilon \right) \leq \frac{1}{2}%
\right) \\
&=&\left( I\right) +\left( II\right)
\end{eqnarray*}%
To deal with $\left( II\right) $, choose $\varepsilon =c_{2}\left\vert
h\right\vert _{\rho \text{-var};\left[ s,t\right] }$ and then $c_{2}$ large
enough so that $c_{1}\psi \left( 1/c_{2}\right) \leq 1/2$. For this $c_{2}$
we then have $\left( II\right) \leq c_{2}\left\vert h\right\vert _{\rho 
\text{-var};\left[ s,t\right] }$. To deal with $\left( I\right) $, we choose 
$\varepsilon =$ $\left\Vert \mathbf{y}\right\Vert _{\psi -\text{var;}\left[
s,t\right] }/\eta $ and using the assumption on $\Delta _{\eta }$%
\begin{eqnarray*}
&\leq &c_{3}\inf \left( \varepsilon >0:V_{\psi \text{;}\left[ s,t\right]
}\left( \delta _{1/\varepsilon }\mathbf{y}\right) \leq 1\right) \\
&&\vee c_{2}\left\vert h\right\vert _{\rho \text{-var};\left[ s,t\right] } \\
&\lesssim &\left\Vert \mathbf{y}\right\Vert _{\psi -\text{var;}\left[ s,t%
\right] }+\left\vert h\right\vert _{\rho \text{-var;}\left[ s,t\right] }.
\end{eqnarray*}
\end{proof}

\textbf{Acknowledgement:} The authors would like to thank the Mittag-Leffler
Institute for its hospitality. They are indebted to Nicolas Victoir and
Terry Lyons for helpful conversations.

\bibliographystyle{plain}
\bibliography{roughpaths}

\def\cprime{$'$} \def\cprime{$'$}
\begin{thebibliography}{10}

\bibitem{benarous-89}
G{\'e}rard Ben~Arous.
\newblock Flots et s\'eries de {T}aylor stochastiques.
\newblock {\em Probab. Theory Related Fields}, 81(1):29--77, 1989.

\bibitem{coutin-qian-02}
Laure Coutin and Zhongmin Qian.
\newblock Stochastic analysis, rough path analysis and fractional {B}rownian
  motions.
\newblock {\em Probab. Theory Related Fields}, 122(1):108--140, 2002.

\bibitem{dudley-norvaisa-99}
Richard~M. Dudley and Rimas Norvai{\v{s}}a.
\newblock {\em Differentiability of six operators on nonsmooth functions and
  {$p$}-variation}, volume 1703 of {\em Lecture Notes in Mathematics}.
\newblock Springer-Verlag, Berlin, 1999.
\newblock With the collaboration of Jinghua Qian.

\bibitem{friz-victoir-05}
Peter Friz and Nicolas Victoir.
\newblock Approximations of the {B}rownian rough path with applications to
  stochastic analysis.
\newblock {\em Ann. Inst. H. Poincar\'e Probab. Statist.}, 41(4):703--724,
  2005.

\bibitem{friz-victoir-05-JFA}
Peter Friz and Nicolas Victoir.
\newblock A variation embedding theorem and applications.
\newblock {\em J. Funct. Anal.}, 239(2):631--637, 2006.

\bibitem{friz-victoir-07-DEdrivenGaussian_I}
Peter Friz and Nicolas Victoir.
\newblock Differential equations driven by {G}aussian signals {I}.
\newblock arXiv:0707.0313, 2007.

\bibitem{friz-victoir-06-Euler}
Peter Friz and Nicolas Victoir.
\newblock Euler estimates for rough differential equations.
\newblock {\em Accepted, Journal of Differential Equations}, 2007.

\bibitem{JainMonrad83}
Naresh~C. Jain and Ditlev Monrad.
\newblock Gaussian measures in {$B\sb{p}$}.
\newblock {\em Ann. Probab.}, 11(1):46--57, 1983.

\bibitem{KawadaKono73}
Takayuki Kawada and Norio K{\^o}no.
\newblock On the variation of {G}aussian processes.
\newblock In {\em Proceedings of the Second Japan-USSR Symposium on Probability
  Theory (Kyoto, 1972)}, pages 176--192. Lecture Notes in Math., Vol. 330,
  Berlin, 1973. Springer.

\bibitem{ledoux-1996}
Michel Ledoux.
\newblock Isoperimetry and {G}aussian analysis.
\newblock In {\em Lectures on probability theory and statistics (Saint-Flour,
  1994)}, volume 1648 of {\em Lecture Notes in Math.}, pages 165--294.
  Springer, Berlin, 1996.

\bibitem{LesniewiczOrlicz73:OngeneralizedVariationsII}
R.~L{\'e}sniewicz and W.~Orlicz.
\newblock On generalized variations. {II}.
\newblock {\em Studia Math.}, 45:71--109, 1973.

\bibitem{lyons-98}
Terry Lyons.
\newblock Differential equations driven by rough signals.
\newblock {\em Rev. Mat. Iberoamericana}, 14(2):215--310, 1998.

\bibitem{lyons-qian-02}
Terry Lyons and Zhongmin Qian.
\newblock {\em System {C}ontrol and {R}ough {P}aths}.
\newblock Oxford University Press, 2002.
\newblock Oxford Mathematical Monographs.

\bibitem{MusielakOrlicz59:OngeneralizedVariationI}
J.~Musielak and W.~Orlicz.
\newblock On generalized variations. {I}.
\newblock {\em Studia Math.}, 18:11--41, 1959.

\bibitem{Neuenschwander96:ProbaonHeisenbergGroup}
D.~Neuenschwander.
\newblock {\em Probabilities on the {H}eisenberg group}, volume 1630 of {\em
  Lecture Notes in Mathematics}.
\newblock Springer-Verlag, Berlin, 1996.
\newblock Limit theorems and Brownian motion.

\bibitem{SaSt91}
L.~Saloff-Coste and D.~W. Stroock.
\newblock Op\'erateurs uniform\'ement sous-elliptiques sur les groupes de
  {L}ie.
\newblock {\em J. Funct. Anal.}, 98(1):97--121, 1991.

\bibitem{St88}
Daniel~W. Stroock.
\newblock Diffusion semigroups corresponding to uniformly elliptic divergence
  form operators.
\newblock In {\em S\'eminaire de Probabilit\'es, XXII}, volume 1321 of {\em
  Lecture Notes in Math.}, pages 316--347. Springer, Berlin, 1988.

\bibitem{Taylor72}
S.~J. Taylor.
\newblock Exact asymptotic estimates of {B}rownian path variation.
\newblock {\em Duke Math. J.}, 39:219--241, 1972.

\end{thebibliography}

\end{document}